\documentclass[12pt]{amsart}
\usepackage{amsmath,amssymb,amsthm,amscd}

\usepackage{graphicx}

\newtheorem{formula}{}[section]
\newtheorem{proposition}[formula]{Proposition}
\newtheorem{corollary}[formula]{Corollary}
\newtheorem{lemma}[formula]{Lemma}
\newtheorem{theorem}[formula]{Theorem}

\newtheorem{question}{Question}

\theoremstyle{definition}
\newtheorem{definition}[formula]{Definition}
\newtheorem{construction}[formula]{Construction}
\newtheorem{example}[formula]{Example}

\theoremstyle{remark}
\newtheorem{remark}[formula]{Remark}

\begin{document}

\title[Hyperelliptic four-manifolds defined by vector-colorings of simple polytopes]{Hyperelliptic four-manifolds defined by vector-colorings of simple polytopes}
\author[Nikolai Erokhovets]{Nikolai Erokhovets}
\address{National Research University Higher School of Economics, Moscow\& Department of Mechanics and Mathematics, Lomonosov Moscow State University}
\email{erochovetsn@hotmail.com}

\def\sgn{\mathrm{sgn}\,}
\def\bideg{\mathrm{bideg}\,}
\def\tdeg{\mathrm{tdeg}\,}
\def\sdeg{\mathrm{sdeg}\,}
\def\grad{\mathrm{grad}\,}
\def\ch{\mathrm{ch}\,}
\def\sh{\mathrm{sh}\,}
\def\th{\mathrm{th}\,}

\def\mod{\mathrm{mod}\,}
\def\In{\mathrm{In}\,}
\def\Im{\mathrm{Im}\,}
\def\Ker{\mathrm{Ker}\,}
\def\Hom{\mathrm{Hom}\,}
\def\Tor{\mathrm{Tor}\,}
\def\rk{\mathrm{rk}\,}
\def\codim{\mathrm{codim}\,}

\def\ko{{\mathbf k}}
\def\sk{\mathrm{sk}\,}
\def\RC{\mathrm{RC}\,}
\def\gr{\mathrm{gr}\,}

\def\R{{\mathbb R}}
\def\C{{\mathbb C}}
\def\Z{{\mathbb Z}}
\def\A{{\mathcal A}}
\def\B{{\mathcal B}}
\def\K{{\mathcal K}}
\def\M{{\mathcal M}}
\def\N{{\mathcal N}}
\def\E{{\mathcal E}}
\def\G{{\mathcal G}}
\def\D{{\mathcal D}}
\def\F{{\mathcal F}}
\def\L{{\mathcal L}}
\def\V{{\mathcal V}}
\def\H{{\mathcal H}}

\thanks{The study has been funded within the framework of the HSE University Basic Research Program}





\begin{abstract}
Toric topology assigns to each simple  convex $n$-polytope $P$ with $m$ facets an~$n$-dimensional real moment angle 
manifold $\mathbb R\mathcal{Z}_P$ with a canonical action of $\mathbb Z_2^m=(\mathbb Z/2\mathbb Z)^m$.
We consider  (non-necessarily free) actions of subgroups $H\subset \mathbb Z_2^m$ 
on  $\mathbb R\mathcal{Z}_P$.  
The orbit space $N(P,H)=\mathbb R\mathcal{Z}_P/H$ has an action of~$\mathbb Z_2^m/H$.
For general $n$ we introduce the notion of a~Hamiltonian $\mathcal{C}(n,k)$-subcomplex in the boundary of 
an $n$-polytope $P$ generalizing the notions of a Hamiltonian cycle  ($k=2$), 
Hamiltonian  theta-subgraph ($k=3$) and Hamiltonian $K_4$-subgraph ($k=4)$ in the $1$-skeleton of a $3$-polytope.
Each $\mathcal{C}(n,k)$-subcomplex $C\subset \partial P$ corresponds to~a~subgroup  
$H_C\subset\mathbb Z_2^m$ such that $N(P,H_C)\simeq S^n$. We prove that in~dimensions $n\leqslant 4$ this
correspondence is~a~bijection. Any~subgroup $H\subset \mathbb Z_2^m$ defines a complex $\mathcal{C}(P,H)\subset \partial P$. 
We prove that each Hamiltonian $\mathcal{C}(n,k)$-subcomplex $C\subset \mathcal{C}(P,H)$ inducing $H$
corresponds to~a~hyperelliptic involution $\tau_C\in\mathbb Z_2^m/H$ on~the~manifold $N(P,H)$ 
(that is, an~involution with the orbit space homeomorphic to $S^n$) 
and in dimensions $n\leqslant 4$ this correspondence is a bijection.
We prove that for the geometries $\mathbb X= \mathbb S^4$, $\mathbb S^3\times\mathbb R$, $\mathbb S^2\times \mathbb S^2$, 
$\mathbb S^2\times \mathbb R^2$, $\mathbb S^2\times \mathbb L^2$, and $\mathbb L^2\times \mathbb L^2$
there exists a compact right-angled $4$-polytope $P$ with a free action of $H$
such that the~geometric manifold  $N(P,H)$ has a~hyperelliptic involution in~$\mathbb Z_2^m/H$, 
and for $\mathbb X=\mathbb R^4$, $\mathbb L^4$, $\mathbb L^3\times \mathbb R$ and $\mathbb L^2\times \mathbb R^2$ 
there are no such polytopes.
\end{abstract}
\maketitle
\tableofcontents
\setcounter{section}{0}
\section{Introduction}

The paper is motivated by the following question asked by A.D.~Mednykh:
{\it for which four-dimensional geometries there is a~geometric hyperelliptic manifold obtained by a generalization of 
the $3$-dimensional construction by A.D.~Mednykh and A.Yu.~Vesnin \cite{M90, VM99M, VM99S2} obtained in~\cite{E24}.} 
We give a partial answer to this question. 

For an introduction to the polytope theory we refer to \cite{Gb03, Z95}. We will use definitions and notations
from \cite{E24}, but for convenience try to write them explicitly. In this paper for~topological spaces 
the~notation $X\simeq Y$ means that $X$ and $Y$ are homeomorphic, and for~complexes $C_1\simeq C_2$ 
means that $C_1$ and $C_2$ are equivalent.

Toric topology (see \cite{BP15,DJ91}) assigns to each simple convex $n$-polytope $P$ with $m$ facets  $F_1,\dots, F_m$
the~{\it real moment-angle manifold} 
$$
\mathbb R\mathcal{Z}_P=P\times \mathbb Z_2^m/\sim, 
\text{ where }(p,a)\sim(q,b)\text{ if and only if }p=q \text{ and }a-b\in\langle \boldsymbol{e}_i\colon p\in F_i\rangle,
$$ 
and $\boldsymbol{e}_1,\dots, \boldsymbol{e}_m$ is the standard basis in $\mathbb Z_2^m$.
$\mathbb R\mathcal{Z}_P$ is a smooth manifold with a~smooth action of~$\mathbb Z_2^m$ 
such that $\mathbb R\mathcal{Z}_P/\mathbb Z_2^m=P$.

We consider  (non-necessarily free) actions of subgroups $H\subset \mathbb Z_2^m$ on  $\mathbb R\mathcal{Z}_P$.  
Each subgroup $H$ of rank $m-r$ can be described as~a~kernel of~an~epimorphism $\mathbb Z_2^m\to\mathbb Z_2^r$
mapping the basis vector $\boldsymbol{e}_i$ to $\Lambda_i\in\mathbb Z_2^r$.
\begin{definition}
We call a mapping $\Lambda\colon \{F_1,\dots,F_m\}\to \mathbb Z_2^r$, $F_i\to \Lambda_i$,
such that $\langle \Lambda_1,\dots,\Lambda_m\rangle=\mathbb Z_2^r$ a {\it vector-coloring of $P$ of rank $r$}.
The subgroup  $H$ is~uniquely defined by $\Lambda$, while $\Lambda$ is~defined up~to
a~linear automorphism of $\mathbb Z_2^r$. A vector-coloring is {\it linearly independent} if for each 
face $F_{i_1}\cap \dots\cap F_{i_k}\ne\varnothing$ the vectors $\Lambda_{i_1}$, $\dots$, $\Lambda_{i_k}$
are linearly independent. Each vector-coloring $\Lambda$ of~rank~$r$ defines the orbit space
$N(P,\Lambda)=\mathbb R\mathcal{Z}_P/H$ with an action of $\mathbb Z_2^r\simeq \mathbb Z_2^m/H$.
\end{definition}
It can be shown that the action of $H$ is free if and only if $\Lambda$ is linearly independent. 
In this case $N(P,\Lambda)$ is a smooth manifold. In particular, 
linearly independent vector-colorings of rank $n$ 
correspond to {\it small covers} $N(P,\Lambda)$ introduced in \cite{DJ91}. 

Linearly independent-vector colorings also arise in~the following construction of geometric
manifolds developed in~the~papers by A.Yu.~Vesnin and A.D.~Mednykh in~\cite{MV86, V87, M90, VM99M, V17}.
\begin{construction}\label{MVconstr}
Let $P$ be a compact right-angled polytope in some geometry $\mathbb X$ (in this paper we will 
use geometries of the form $\mathbb X= \mathbb X_1\times\dots\times \mathbb X_k$, where $\mathbb X_i\in \{\mathbb S^{n_i},\mathbb R^{n_i},\mathbb L^{n_i}\}$ and by~a~right-angled polytope we mean the product of the corresponding right-angled polytopes). 

The polytope $P$ corresponds to a right-angled Coxeter group 
$$
\langle \rho_1,\dots,\rho_m\rangle/(\rho_1^2,\dots,\rho_m^2, \rho_i\rho_j=\rho_j\rho_i,
\text{ if }F_i\cap F_j\ne\varnothing).
$$
This group is isomorphic to a subgroup $G(P)$ of isometries of~$\mathbb X$ generated by~reflections in~facets of~$P$,
where $\rho_i$ corresponds to~the~reflection in~$F_i$. The group $G(P)$ acts on~$\mathbb X$ discretely, 
$P$~is~a~fundamental domain and the orbit space, and for any point of~$P$ its stabilizer is generated by reflections in facets containing this point (see \cite[Theorem 1.2 in Chapter 5]{VS88}).

A linearly independent vector coloring $\Lambda$ of rank $r$ defines 
the~epimorphism $\varphi_{\Lambda}\colon G(P)\to \mathbb Z_2^r$
by the rule $\varphi_{\Lambda}(\rho_i)=\Lambda(F_i)$.  
The subgroup ${\rm Ker\,} \varphi_{\Lambda}$ acts freely on $\mathbb X$, 
and the quotient space $\mathbb X/{\rm Ker\,} \varphi_{\Lambda}$ is~a~closed manifold with the~geometric structure modelled on $\mathbb X$. It is easy to 
see that $N(P,\Lambda)\simeq \mathbb X/{\rm Ker\,} \varphi_{\Lambda}$ (see more details in \cite[Construction 4.11]{E22M}).
\end{construction}

For a~general (not necessarily linearly independent) vector-coloring $\Lambda$
is was proved in \cite[Theorem 5.1]{E24} (jointly with D.V.~Gugnin) that $N(P,\Lambda)$ is a closed
topological manifold if and only if for each face $F_{i_1}\cap \dots\cap F_{i_k}\ne\varnothing$ {\bf different}
vectors among $\{\Lambda_{i_1}, \dots, \Lambda_{i_k}\}$ are linearly independent.
(This result also can be extracted from the general results 
 by~M.A.~Mikhailova and C.~Lange \cite{M85,LM16,L19}.) This gives rise to the following definition. 
\begin{definition}\cite[Definition 2.1]{E24}
A coloring $c$ of a simple polytope $P$ in~$l$ colors is~a~surjective mapping from 
$\{F_1,\dots, F_m\}$ to~a~finite set of $l$ elements. For convenience we identify the~set with 
$[l]=\{1,\dots, l\}$, but in this paper often this is a subset in $\mathbb Z_2^r$. 
For any coloring $c$ define a complex $\mathcal{C}(P,c)\subset \partial P$ as~follows. 
Its ``facets'' are connected components of~unions of~all~the~facets of~$P$ of the same color, 
``$k$-faces'' are connected components of~intersections of~$(n-k)$ different facets. By definition each $k$-face
is a union of $k$-faces of $P$. We choose a~linear order of all the facets $G_1$, $\dots$, $G_M$. Two 
complexes $\mathcal{C}(P,c)$ and $\mathcal{C}(Q,c')$ are  {\it equivalent} (we write $\mathcal{C}(P,c)\simeq \mathcal{C}(Q,c')$) 
if there is~a homeomorphism $P\to Q$ sending facets of $\mathcal{C}(P,c)$ to~facets of~$\mathcal{C}(Q,c')$. 
\end{definition}
Each $k$-face of~$\mathcal{C}(P,c)$ is a~connected orientable $k$-manifold, perhaps with a~boundary
(see Lemma~\ref{lem:or}). The~closed manifold $N(P,\Lambda)$ is orientable 
if and only if for some change of coordinates in $\mathbb Z_2^r$ we have $\Lambda_i=(1,\lambda_i)$ 
(see~\cite[Corollary 1.15]{E24}). We call the~mapping $\lambda\colon F_i\to \lambda_i$ an~{\it affine coloring} of~$P$ 
of~rank $(r-1)$.

It can be shown (see \cite[Proposition 2.6]{E24}) that for any coloring $c$ of the simplex  $\Delta^n$ in $k$ colors
the~complex $\mathcal{C}(\Delta^n,c)$ is equivalent to the complex given on 
$$
S^n_{k,\geqslant 0}=\{x_1^2+\dots+x_{n+1}^2=1, x_1\geqslant 0,\dots, x_k\geqslant 0\}
$$ 
by the facets $S^n_{k,\geqslant}\cap \{x_i=0\}$, $i=1,\dots, k$, and (\cite[Corollary 2.11]{E24}) to the complex given on 
$$
B^n_{k-1,\geqslant}=\{(x_1,\dots,x_n)\in\mathbb R^n\colon x_1\geqslant 0,\dots, x_{k-1}\geqslant 0, x_1^2+\dots+x_n^2\leqslant 1\}
$$
by the facets  $B^n_{k-1,\geqslant}\cap\{x_s=0\}$, $s=1,\dots, k-1$, and  $B^n_{k-1,\geqslant}\cap\{x_1^2+\dots+x_n^2=1\}$.
We denote the~equivalence class of these complexes $\mathcal{C}(n,k)$.
For $n=3$ the complex $\mathcal{C}(3,1)$ has a~single facet $\partial P$, the complex 
$\mathcal{C}(3,1)$ corresponds a simple edge-cycle in $\partial P$, $\mathcal{C}(3,2)$ -- to~a~theta-subgraph in~$\partial P$
consisting of two different vertices of $P$ and three disjoint (outside these vertices) simple edge-paths connecting them,
and $\mathcal{C}(3,4)$ corresponds to a~$K_4$-subgraph in~$\partial P$ consisting of $4$ disjoint vertices of $P$ 
pairwise connected by a set of disjoint simple edge-paths. 

It is easy to see that if for a vector coloring of rank $r$ the complex $\mathcal{C}(P,\Lambda)$ 
is equivalent to~$\mathcal{C}(n,r)$, then $N(P,\Lambda)\simeq S^n$  (see  \cite[Construction 5.8]{E24}). 
The converse trivially holds for~$n=2$. It was proved in \cite[Theorem 10.1]{E24} that the converse 
also holds for $n=3$. {\bf The first main result} (Theorem \ref{th:n34}) of~this paper states that this holds for $n=4$. 
Namely, {\it for~a~vector-coloring $\Lambda$ of~rank~$r$ of~a~simple polytope of dimension $n\leqslant 4$ we have 
$N(P,\Lambda)\simeq S^n$ if and only if $\mathcal{C}(P,\Lambda)\simeq \mathcal{C}(n,r)$.} For general $n$
we prove (Lemma \ref{lem:RHDS} and Corollary \ref{cor:NPSn}) that {\it if~$N(P,\Lambda)\simeq S^n$, then for 
each $\omega=\omega_1\sqcup \dots\sqcup\omega_k\subset[M]$ 
the set $\left(\bigcup\limits_{i_1\in \omega_1} G_{i_1}\right)\cap\dots
\cap \left(\bigcup\limits_{i_k\in \omega_k} G_{i_k}\right)$ is~a~ rational homology $(n-k)$-disk if~$\omega\ne[M]$, and
a~rational homology $(n-k)$-sphere if~$\omega=[M]$. 
In~particular,  $M\leqslant n+1$ and each $k$-face of the complex $\mathcal{C}(P,\Lambda)$
is either a $k$-$RHD$ or a $k$-$RHS$ (by definition we assume that for $k<0$
a~$k$-$RHD$ is a point, and a~$k$-$RHS$ is empty).}

On the base of these results we study hyperelliptic involutions in the group $\mathbb Z_2^r$ acting 
on~the~manifold $N(P,\Lambda)$.
\begin{definition}
Following \cite{VM99S1} we call a closed $n$-manifold $M$ {\it hyperelliptic} if it has an~involution $\tau$ such that
$M/\langle\tau\rangle\simeq S^n$. The corresponding involution $\tau$ is~called a~{\it hyperelliptic involution}.
\end{definition}
 
In dimension $n=3$ in the papers \cite{M90, VM99M, VM99S2} A.D.~Mednykh and A.Yu.Vesnin
constructed examples of~hyperelliptic $3$-manifolds
in five of eight Thurston's geometries: $\mathbb S^3$, $\mathbb R^3$, $\mathbb L^3$, $\mathbb S^2\times\mathbb R$, and 
$\mathbb L^2\times\mathbb R$. Each manifold was constructed using a~Hamiltonian cycle, a~Hamiltonian theta-subgraph
or~a~Hamiltonian $K_4$-subgraph in~the $1$-skeleton of a~compact 
right-angled $3$-polytope, where a~subgraph is~{\it Hamiltonian},
if it contains all vertices of the polytope. In \cite[Theorem 11.5]{E24} it was proved that for a~$3$-manifold $N(P,\Lambda)$
over a~$3$-polytope $P$ hyperelliptic involutions in~$\mathbb Z_2^r$ are~in~bijection with 
a~Hamiltonian empty set, cycles, theta-subgraphs and $K_4$-subgraphs in the $1$-skeleton 
of~$\mathcal{C}(P,\Lambda)$ inducing $\Lambda$, where 
for~a~Hamiltonian subgraph $\Gamma$ of the above types in the $1$-skeleton of~$\mathcal{C}(P,c)$
there is~a~canonical way to construct a vector-coloring $\widetilde{\Lambda}_{\Gamma}$ of $P$
{\it induced by $\Gamma$} such that $\mathcal{C}(P,c)=\mathcal{C}(P,\widetilde{\Lambda}_{\Gamma})$. 
Also in \cite[Theorem 11.7]{E24} the classification of vector-colorings with 
more than one hyperelliptic involutions was presented. 

In this paper we generalize \cite[Theorem 11.5]{E24} to dimension $n=4$. For this
purpose we introduce the~notion of a {\it Hamiltonian $\mathcal{C}(n,k)$-subcomplex} $C\subset \mathcal{C}(P,c)$
(Definition~\ref{def:CnkH}) and explain (Construction~\ref{con:LHC}) how to build a canonical induced vector-coloring 
$\widetilde{\Lambda}_C$ of rank $k+1$ 
such that $\mathcal{C}(P,c)=\mathcal{C}(P,\widetilde{\Lambda}_C)$ and $N(P,\widetilde{\Lambda}_C)$
has a canonical hyperelliptic involution in $\mathbb Z_2^{k+1}$. Then for a vector-coloring $\Lambda$
of rank $r$ any Hamiltonian $\mathcal{C}(n,r-1)$-subcomplex $C\subset \mathcal{C}(P,\Lambda)$ inducing $\Lambda$
corresponds to a hyperelliptic involution $\tau_C\in \mathbb Z_2^r$. Our {\bf second main result}  
(Theorem~\ref{th:Hamsc}) is that {\it in dimensions $n\leqslant 4$ this correspondence is a bijection}.

{\bf The third main result} (Theorem~\ref{th:geomn=4} and Remark~\ref{rem:4hi}) of this paper
is~the~answer to~the~following question in~dimension $n=4$.
\begin{question}
For which geometries of the form $\mathbb X= \mathbb X_1\times\dots\times \mathbb X_k$, where $\mathbb X_i\in \{\mathbb S^{n_i},\mathbb R^{n_i},\mathbb L^{n_i}\}$ theres exists a right-angled polytope $P$ and a~linearly 
independent vector-coloring $\Lambda$ of rank $r$ such that the geometric manifold $N(P,\Lambda)$ has a hyperelliptic
involution in~the~group~$\mathbb Z_2^r$ canonically acting on it.
\end{question}
In dimension $n=4$ there are $10$ geometries of this form. We prove that {\it for
the geometries $\mathbb S^4$, $\mathbb S^3\times \mathbb R$, 
$\mathbb S^2\times \mathbb S^2$, $\mathbb S^2\times \mathbb R^2$, $\mathbb S^2\times \mathbb L^2$,
and $\mathbb L^2\times \mathbb L^2$ there is such a pair $(P,\Lambda)$, and for 
the~geometries $\mathbb R^4$, $\mathbb L^4$, $\mathbb L^3\times \mathbb R$, $\mathbb L^2\times\mathbb R^2$
there are no such pairs.  }

The paper is organized as follows.

In Section~\ref{sec:RHS} on the base of results from \cite{CP17,CP20} we give a~criterion
(Lemma~\ref{pr:rhs}) when $N(P,\Lambda)$ is a rational homology $n$-sphere and specify~it 
(Theorem~\ref{th:rhs34}) for~dimensions $n\leqslant 4$.

In Section~\ref{sec:Sn4} we prove our {\bf first main result} (Theorem~\ref{th:n34}) -- 
a criterion when $N(P,\Lambda)\simeq S^n$ for $n\leqslant 4$. For the~proof we~use
the~Armstrong theorem and results from Section~\ref{sec:RHS}. For general $n$
we give a necessary condition (Lemma~\ref{lem:RHDS} and Corollary~\ref{cor:NPSn}).

In Section~\ref{sec:hyp} we study hyperelliptic involutions in $\mathbb Z_2^r$ acting 
on manifolds $N(P,\Lambda)$. 

In Subsection~\ref{ssec:Hams} we introduce 
the notions of a~{\it subcomplex} $\mathcal{C}(P,c_2)\subset \mathcal{C}(P,c_1)$ (Construction~\ref{con:subcom})
and a~{\it Hamiltonian subcomplex} (Definition~\ref{def:Ham}). We give a~criterion  (Lemma~\ref{lem:Hamsc})
when the subcomplex $\mathcal{C}(P,c_2)\subset \mathcal{C}(P,c_1)$  is Hamiltonian. We introduce 
(Definition \ref{def:n-2f}) the~notion of~{\it defining $(n-2)$-faces} of a subcomplex and 
reformulate (Proposition~\ref{cor:n-2f}) the criterion in~terms of~the~defining faces. In 
Proposition~\ref{prop:critdef} we give a criterion when a set of disjoint $(n-2)$-faces is a set of defining
faces of a Hamiltonian subcomplex.

In Subsection~\ref{ssec:bip} we introduce (Definition~\ref{def:adjg}) the~notion of~the
{\it adjacency graph} of a Hamiltonian subcomplex $\mathcal{C}(P,c_2)\subset \mathcal{C}(P,c_1)$.
If this graph is bipartite we call the Hamiltonian subcomplex {\it bipartite} (Definition~\ref{def:bip}).
\begin{itemize}
\item In Construction~\ref{con:indfrom} we show how to~induce the~vector-coloring from 
a~bipartite Hamiltonian subcomplex and in Example~\ref{ex:vci} we explain how a~bipartite Hamiltonian
subcomplex $C$ defines a canonical vector-coloring {\it induced by}~$C$. 
\item  In~Proposition~\ref{propo:taubpH}
we show that for a~closed manifold $N(P,\Lambda)$ there is a bijection 
between the involutions in $\mathbb Z_2^r\setminus\{0\}$ and proper bipartite
Hamiltonian subcomplexes $C\subset \mathcal{C}(P,\Lambda)$ such that $\Lambda$
is~induced from some vector-coloring of $C$. 
\item In Proposition~\ref{prop:Plman} we~give 
a~criterion when for a closed manifold $N(P,\Lambda)$ and an~involution $\tau\in\mathbb Z_2^r$
the orbit space $N(P,\Lambda)/\langle\tau\rangle$ is also a closed manifold. 
\item In Proposition~\ref{prop:indman}
we prove that if the vector-coloring of a bipartite Hamiltonian subcomplex defines a~closed manifold, 
then so does the~induced vector-coloring.
\end{itemize}

In Subsection~\ref{ssec:4c} using the notion
of a bipartite Hamiltonian subcomplex we reformulate (Corollary~\ref{cor:4cp}) the~four color theorem
for a simple $3$-polytope $P$ in terms of a disjoint union of simple edge-cycles containing all the vertices of $P$.

In Subsection~\ref{ssec:faces} we study the structure of faces of the complex $\mathcal{C}(P,c)$ 
(Lemmas \ref{lem:cc} and \ref{lem:intG}, Corollary~\ref{cor:fG}). 
\begin{itemize}
\item In Proposition~\ref{prop:Gk} for a closed
manifold $N(P,\Lambda)$ and a face $G$ of~$\mathcal{C}(P,\Lambda)$ we describe 
the~submanifold in $N(P,\Lambda)$ arising as a preimage of $G$ under the~projection $N(P,\Lambda)\to P$. 
\item In Proposition~\ref{prop:kdefk+1} we describe the face structure of defining faces $M_q$ of a~Hamiltonian subcomplex
$C\subset \mathcal{C}(P,c)$ in~terms of~the~intersections of~$M_q$ with faces of~$C$.
\item  In Proposition~\ref{prop:Gtilde}
and Corollary~\ref{cor:GneMq} we describe $k$-faces of a Hamiltonian subcomplex $C\subset \mathcal{C}(P,c)$ 
as~unions of~$k$-faces of~$\mathcal{C}(P,c)$ separated by $(k-1)$-faces of defining faces.  
\end{itemize}

In Subsection~\ref{ssec:bc} for a closed manifold $N(P,\Lambda)$ and an involution
$\tau\in\mathbb Z_2^r$ such that $N(P,\Lambda)/\langle\tau\rangle$ is also a closed manifold
we describe the mapping $N(P,\Lambda)\to N(P,\Lambda)/\langle\tau\rangle$ as~a~two-sheeted branched covering.

In Subsection~\ref{ssec:cnk} we introduce (Definition~\ref{def:CnkH}) the notions 
of~a~$\mathcal{C}(n,k)$-{\it subcomplex} and~a~{\it Hamiltonian $\mathcal{C}(n,k)$-subcomplex}. 
\begin{itemize}
\item In Corollary~\ref{cor:HCnk} we describe explicitly the subgroup $H(C)\subset\mathbb Z_2^m$ corresponding
to~a~$\mathcal{C}(n,k)$-subcomplex $C\subset \partial P$. 
\item In Corollary~\ref{cor:cnk4} we reformulate 
our {\bf first main result} in terms of  \linebreak$\mathcal{C}(n,k)$-subcomplexes $C\subset \partial P$
of a polytope $P$ of dimension $n\leqslant 4$. 
\item In Proposition~\ref{prop:cnkbh}
we prove that for any facet $\widetilde{G}$ of a~proper Hamiltonian $\mathcal{C}(n,k)$-subcomplex $C\subset\mathcal{C}(P,c)$
the adjacency graph of facets of $\mathcal{C}(P,c)$ lying in $\widetilde{G}$ is a~tree. In~particular, 
any proper Hamiltonian $\mathcal{C}(n,k)$-subcomplex is bipartite.
\item In Construction~\ref{con:LHC} we show that any proper Hamiltonian 
$\mathcal{C}(n,k)$-subcomplex $C\subset\mathcal{C}(P,c)$ induces a canonical vector-coloring 
$\widetilde{\Lambda}_C$~of rank $(k+1)$ such that $\mathcal{C}(P,c)=\mathcal{C}(P,\widetilde{\Lambda}_C)$
and $N(P,\widetilde{\Lambda}_C)$ is a hyperelliptic manifold with a canonical hyperelliptic involution
in $\mathbb Z_2^{k+1}$ corresponding to $C$.
\item  In~Theorem~\ref{th:Hamsc} we formulate our 
{\bf second main result}, namely that for $n\leqslant 4$ and a vector-coloring of rank $r$ 
of~a~simple $n$-polytope $P$ such that $N(P,\Lambda)$ is a closed manifold
hyperelliptic involutions in $\mathbb Z_2^r$ are in bijection with proper Hamiltonian $\mathcal{C}(n,r-1)$-subcomplexes 
$C\subset\mathcal{C}(P,\Lambda)$ inducing $\Lambda$. 
\item In Proposition~\ref{prop:treek} we prove
that for a~proper Hamiltonian $\mathcal{C}(n,r)$-subcomplex $C\subset \mathcal{C}(P,c)$ if~a~$k$-face 
$\widetilde{G}$ of $C$ is not a circle, then the~adjacency graph $\Gamma(\widetilde{G})$ of~$k$-faces 
$G$ of~$\mathcal{C}(P,c)$ lying in~$\widetilde{G}$ is~a~tree.
\item  Using this fact in Proposition~\ref{prop:fkform}
we obtain the formula expressing the number of $k$-faces of~a~complex $\mathcal{C}(P,c)$
in~terms of~the~numbers of~$k$-faces and $(k-1)$-faces of~defining faces 
of~a~proper~Hamiltonian $\mathcal{C}(n,r)$-subcomplex.
\item  In Proposition~\ref{prop:Mq} we prove that for any defining
$(n-2)$-face $M_q$ of a proper Hamiltonian $\mathcal{C}(n,r)$-subcomplex $C\subset\mathcal{C}(P,c)$
its facets can be colored in $(r-1)$ colors in such a way that adjacent facets have different colors.
\end{itemize}

In Subsection~\ref{ssec:geom} we study geometric structures on hyperelliptic manifolds $N(P,\Lambda)$.
Our main tool is Construction~\ref{con:hypgeom} of a~geometric hyperelliptic manifold using 
a proper~Hamiltonian $\mathcal{C}(n,k)$-subcomplex $C$ in the boundary of a compact right-angled polytope $P$
in some geometry $\mathbb X$. This construction is a direct generalization of the construction
from \cite{M90, VM99M, VM99S2} based on a Hamiltonian cycle, theta- or $K_4$-subgraph 
in~the~boundary of~a~right-angled $3$-polytope. 
\begin{itemize}
\item In Corollary~\ref{cor:In} we show that $\mathbb I^n$
does not admit Hamiltonian \linebreak $\mathcal{C}(n,k)$-subcomplexes for~$n\leqslant 4$. In particular, the geometry 
$\mathbb R^n$ does not admit Construction~\ref{con:hypgeom} for $n\geqslant 4$.
\item In~Corollaries~\ref{cor:n-2n-2} 
and~\ref{cor:nn-1} and Remark~\ref{rem:nn} we show that if $P$ admits a~Hamiltonian 
$\mathcal{C}(n,n-1)$-subcomplex, then each defining
face is a polytope with only even-gonal $2$-faces. 
\item In Proposition~\ref{prop:Ln} we show
that if a $4$-polytope $P$ admits a~proper Hamiltonian $\mathcal{C}(4,r)$-subcomplex, then $P$
has at least one $2$-face with three or four edges. In particular, the geometry $\mathbb L^4$
does not admit Construction~\ref{con:hypgeom}, as well as $\mathbb L^n$ for $n\geqslant 5$
since in~the~latter geometries there are no compact right-angled polytopes due to results by~V.V.~Nikulin~\cite{N82} 
(see~\cite{V17}). 
\item In~Proposition~\ref{prop:IX} we show that if a $4$-polytope $P=Q\times I$
admits a~proper Hamiltonian $\mathcal{C}(4,r)$-subcomplex, then one of the defining faces is a triangle.
In~particular, the~geometries $\mathbb L^3\times \mathbb R$ and $\mathbb L^2\times \mathbb R^2$ do not admit
Construction~\ref{con:hypgeom}. 
\item In Proposition~\ref{prop:Sn} we show that $\Delta^n$ admits 
a~Hamiltonian $\mathcal{C}(n,n)$-subcomplex. In~particular, the~geometry $\mathbb S^n$ admits 
Construction~\ref{con:hypgeom} for $n\geqslant 2$. 
\item In Proposition~\ref{prop:Sn-1R} we show that 
the prism $\Delta^{n-1}\times I$ admits Hamiltonian~$\mathcal{C}(n,n)$ and~$\mathcal{C}(n,n+1)$-subcomplexes 
for~$n\geqslant 4$. In particular, the geometry $\mathbb S^{n-1}\times\mathbb R$ admits 
Construction~\ref{con:hypgeom} for $n\geqslant 3$.
\item  In Construction~\ref{con:cutF} we describe
the~operation of~cutting off a~face and show that the~initial polytope defines a~proper Hamiltonian
subcomplex in the resulting polytope. As~a~result in~Corollary~\ref{cor:spsq} we show that 
the~polytope $P^n=\Delta^p\times \Delta^q$
has a~Hamiltonian $\mathcal{C}(n,n+1)$-subcomplex for $p,q\geqslant 1$. In particular, 
the geometry $\mathbb S^p\times\mathbb S^q$, $p,q\geqslant 2$, admits Construction~\ref{con:hypgeom}.  
\item In Proposition~\ref{prop:Dn-1} we show that if  $P^n$ has a~Hamiltonian $\mathcal{C}(n,n+1)$-subcomplex such
that its defining faces $\{M_1,\dots, M_s\}$ do not intersect a facet $F_i\simeq \Delta^{n-1}$,
then for any $k\geqslant 1$ the polytope $\Delta^k\times P$ has a~Hamiltonian 
$\mathcal{C}(n+k,n+k+1)$-subcomplex, and in~Corollary~\ref{cor:pqI} we show that
for any $p,q\geqslant 1$ the~polytope $P^n=\Delta^p\times\Delta^q\times I$
has a~Hamiltonian $C(n, n+1)$-subcomplex. 
In~particular, the~geometries $\mathbb S^p\times \mathbb S^q\times \mathbb R$, $p,q\geqslant 2$, and 
$\mathbb S^p\times \mathbb R^2$, $p\geqslant 2$, admit~Construction~\ref{con:hypgeom}.
\item In Proposition~\ref{prop:nn} we show that if an $n$-polytope $P$ has a~Hamiltonian $\mathcal{C}(n,n+1)$-subcomplex $C$ such
that its defining faces $\{M_1,\dots, M_s\}$ do not intersect an $(n-2)$-face $G\subset P$, then $G\simeq \Delta^{n-2}$,
and $P$ has a~Hamiltonian $\mathcal{C}(n,n)$-subcomplex. Using this fact in Corollary~\ref{cor:pk}
we show that $\Delta^2\times \Delta^{n-2}$ has a Hamiltonian~$\mathcal{C}(n,n)$-subcomplex for~$n\geqslant 3$,
and in Corollary~\ref{cor:SxX2} we show that for $p\geqslant 1$ the polytope $P^n=\Delta^p\times P_k$,
where $P_k$ is~a~$k$-gon, $k\geqslant 3$, admits Hamiltonian $\mathcal{C}(n,n)$- and $\mathcal{C}(n,n+1)$-subcomplexes. 
In~particular, the~geometries $\mathbb S^p\times \mathbb S^2$, $\mathbb S^p\times \mathbb R^2$ and 
$\mathbb S^p\times \mathbb L^2$, $p\geqslant 2$, admit Construction~\ref{con:hypgeom}.

\item In Example~\ref{ex:l2l2} we show that the polytope $P_5\times P_5$ admits  
a~Hamiltonian $\mathcal{C}(4,5)$-subcomplex. In particular,
the geometry $\mathbb L^2\times\mathbb L^2$ admits Construction~\ref{con:hypgeom}. 
\item In Lemma~\ref{lem:RZF} we show that if for a~defining $(n-2)$-face $F$ of~a~Hamiltonian $\mathcal{C}(n,k)$-subcomplex $C\subset \partial P$ different facets of $F$ correspond to different facets of $C$, then 
in~the~branch set of~the~covering $N(P,\widetilde{\Lambda}_C)\to S^n$ the~preimage of~$F$ 
is~a~disjoint union of $2^{k-1-m_F}$ copies of~$\mathbb R\mathcal{Z}_F$, where $m_F$ is the number 
of facets of $F$.
\item Using Lemma~\ref{lem:RZF} and Proposition~\ref{prop:Ham2sh} for each 
Hamiltonian $\mathcal{C}(n,k)$-subcomplex $C\subset \partial P$  arising
in this subsection we describe the branch set of the corresponding $2$-sheeted
branched covering of~$S^n$.
\end{itemize}
We conclude Subsection~\ref{ssec:geom}  with Theorem~\ref{th:geomn=4}, 
which together with Remark~\ref{rem:4hi} form our {\bf third main result}.

\section{Rational homology $4$-spheres among $N(P,\Lambda)$}\label{sec:RHS}
\begin{definition}
We call a topological space $X$ a {\it rational homology $n$-sphere} ($n$-$RHS$), 
if~$X$ is~a~compact closed topological $n$-manifold 
and $H_k(X,\mathbb Q)=H_k(S^n,\mathbb Q)$ for all $k$. By definition $X$ is an~$n$-$RHS$
for $n<0$, if $X=\varnothing$.

We call  $X$ a {\it rational homology $n$-disk} ($n$-$RHD$), if $X$ 
is~a~compact orientable topological $n$-manifold with a boundary 
such that $H_k(X,\mathbb Q)=H_k(D^n,\mathbb Q)$ for all $k$, where 
$D^n=\{(x_1,\dots, x_n)\in\mathbb R^n\colon x_1^2+\dots+x_n^2\leqslant 1\}$ is an $n$-disk.
By definition $X$ is an~$n$-$RHD$
for $n<0$, if~$X={\rm pt}$.

\end{definition} 
\begin{lemma}\label{lem:DS}
If $X$ is an $n$-$RHD$, then $\partial X$ is an $(n-1)$-$RHS$ for $n\geqslant 1$.
\end{lemma}
\begin{proof}
This follows from the long exact sequence in homology 
$$
\dots\to \widetilde{H}_k(\partial X)\to \widetilde{H}_k(X)\to H_k(X,\partial X)\to \widetilde{H}_{k-1}(\partial X)\to \dots
$$
the Poincare-Lefschetz duality $H_k(X,\partial X)\simeq H^{n-k}(X)$ and the universal coefficients
formula.
\end{proof}
 
We will use the following result,
which was first proved for small covers and $\mathbb Q$ coefficients in~\cite{ST12, T12}. Let us identify the subsets
$\omega\subset[m]=\{1,\dots,m\}$ with vectors $\boldsymbol{x}\in\mathbb Z_2^m$ by~the~rule $\omega=\{i\colon x_i=1\}$.
For a vector-coloring $\Lambda$ of rank $(r+1)$ denote by ${\rm row}\,\Lambda$ the subspace in $\mathbb Z_2^m$ 
generated by the row vectors of the matrix $\Lambda$. Equivalently, 
$$
{\rm row}\,\Lambda=\{(x_1,\dots,x_m)\in\mathbb Z_2^m\colon\exists\boldsymbol{c}\in(\mathbb Z_2^{r+1})^*\colon x_i=
\boldsymbol{c}\Lambda_i, i=1,\dots,m\}.
$$
Remind that $P_{\omega}=\bigcup_{i\in\omega}F_i$.
\begin{theorem}\cite[Theorem 4.5]{CP17}{\rm (see also \cite[Theorem 1.1]{CP20})}\label{HNPLth}
Let $\Lambda$ be a vector-coloring of rank $(r+1)$ of a simple $n$-polytope $P$ and $R$ be
a commutative ring in which $2$ is a unit. Then there is an $R$-linear isomorphism
$$
H^k(N(P,\Lambda), R)\simeq\bigoplus_{\omega\in{\rm row}\,\Lambda}\widetilde{H}^{k-1}(P_{\omega}, R)
$$
\end{theorem}
Let is remind that a closed orientable manifold $N(P,\Lambda)$ is defined by a~an affine coloring $\lambda$
of rank $r$, where for some change of coordinates in $\mathbb Z_2^{r+1}$ we have $\Lambda_i=(1,\lambda_i)$. 
Theorem~\ref{HNPLth} and the universal coefficients formula imply
\begin{proposition}\label{pr:rhs}
Let $\lambda$ be an affine coloring of rank $r$ of a simple $n$-polytope $P$.
The space $N(P,\lambda)$ is a rational homology $n$-sphere if and only if 
one of~the~following equivalent conditions holds:
\begin{enumerate}
\item $\bigcup\limits_{i\colon \lambda_i\in \pi }F_i$ is~an $(n-1)$-$RHD$~for any affine hyperplane $\pi\subset\mathbb Z_2^r$;
\item  $\bigcup\limits_{i\colon \lambda_i\in \pi }F_i$ is~an $(n-1)$-$RHD$~for any affine hyperplane $\pi\subset\mathbb Z_2^r$
passing through some  pint $\boldsymbol{p}\in\mathbb Z_2^r$.
\end{enumerate}
\end{proposition}

We will use the following {\it Generalized Shoenflies theorem}.

\begin{theorem}\cite[Theorem 5]{B60}\label{th:GST} Let $h$ be a homeomorphic embedding of $S^{n-1}\times I$ into $S^n$.
Then the closure of either complementary domain of $h(S^{n-1}\times\frac{1}{2})$ is homeomorphic to $D^n$.
\end{theorem}
\begin{lemma}\label{lem:rhd4}
Let $n\leqslant 4$. If $P_{\omega}$ is an $(n-1)$-$RHD$, then it is homeomorphic 
to~an~$(n-1)$-disk~$D^{n-1}$. 
\end{lemma}
\begin{proof}
For $n\leqslant 2$ this is trivial. Let $n\in\{3,4\}$. If $P_{\omega}$ is an $(n-1)$-$RHD$, 
then by Lemma~\ref{lem:DS} $\partial P_{\omega}$ is an~$(n-2)$-$RHS$. For $n\in\{3,4\}$ 
this means that $\partial P_{\omega}$ is homeomorphic to~$S^{n-2}$.
Since $P_{\omega}$ and $P_{[m]\setminus\omega}$  are compact piecewise linear manifolds (see Lemma~\ref{lem:Mc}), 
their common
boundary $\partial P_{\omega}$ has a~collar neighbourhood in both spaces and the~embedding  
$S^{n-2}\simeq \partial P_{\omega}$ to $\partial P\simeq S^{n-1}$
can~be~extended to~an~embedding of $S^{n-2}\times I$. Then by~Theorem \ref{th:GST} $P_{\omega}$
and $P_{[m]\setminus\omega}$ are topological $(n-1)$-disks. 
\end{proof}
This leads to the following result.
\begin{theorem}\label{th:rhs34}
Let $\lambda$ be an affine coloring of rank $r$ of a simple $n$-polytope $P$, where $2\leqslant n\leqslant 4$.
The space $N(P,\lambda)$ is a rational homology $n$-sphere if and only if 
one of~the~following equivalent conditions holds:
\begin{enumerate}
\item $\bigcup\limits_{i\colon \lambda_i\in \pi }F_i$ is~an $(n-1)$-disk~for any affine hyperplane $\pi\subset\mathbb Z_2^r$;
\item  $\bigcup\limits_{i\colon \lambda_i\in \pi }F_i$ is~an $(n-1)$-disk~for any affine hyperplane $\pi\subset\mathbb Z_2^r$
passing through some  pint $\boldsymbol{p}\in\mathbb Z_2^r$.
\end{enumerate}
\end{theorem}
\begin{definition}
We call a subset $A\subset B$ {\it proper}, if $\varnothing\ne A\ne B$. For a collection 
of~pairwise disjoint subsets  $\omega_1$, $\dots$, $\omega_k \subset [m]$ define 
$P_{\omega_1,\dots,\omega_k}=P_{\omega_1}\cap P_{\omega_2}\cap\dots\cap P_{\omega_k}$.
\end{definition}
\begin{lemma}\label{lem:or}
For any simple $n$-polytope $P$ and any collection of pairwise disjoint
proper subsets $\omega_1$, $\dots$, $\omega_k \subset [m]$,  if  
$P_{\omega_1,\dots,\omega_k}\ne\varnothing$, then it is~an orientable~$(n-k)$-manifold, perhaps with a~boundary.
\end{lemma}
\begin{proof}
As it is mentioned in \cite[Corollary 2.4]{E24}, $P_{\omega_1,\dots,\omega_k}$ 
is~an~$(n-k)$-manifold, perhaps with a boundary. This is based on the following fact
which we will use later.
\begin{lemma}\label{lem:Mc}\cite[Lemma 2.2]{E24}
Let a point $p\in\partial P$ belong to exactly $l\geqslant 1$ facets $G_{i_1}$, $\dots$, $G_{i_l}$ of $\mathcal{C}(P,c)$. 
Then there is a piecewise linear homeomorphism  $\varphi$ of a neighbourhood $U\subset P$ of $p$ such that 
$U\cap G_j=\varnothing$ for $j\notin \{i_1,\dots, i_l\}$ to a neighbourhood 
$V\subset \mathbb R^l_{\geqslant}\times\mathbb R^{n-l}$ such that $\varphi(G_{j_s}\cap U)=V\cap\{y_s=0\}$, $s=1,\dots,l$.  
\end{lemma}
\begin{remark}
It is not stated explicitly in \cite[Lemma 2.2]{E24} that $U\cap G_j=\varnothing$ for $j\notin \{i_1,\dots, i_l\}$ but this
condition can be be achieved using a~smaller neighbourhood.
\end{remark}
Let us proof that the manifold $P_{\omega_1,\dots,\omega_k}$ is~orientable.
Indeed for $k=1$ the set $P_{\omega}$ is a manifold with a boundary, and $P_{\omega}\setminus\partial P_{\omega}$
is in open subset in $\partial P\simeq S^{n-1}$. Hence, $P_{\omega}$ is orientable.
Assume that our claim is valid for $1,\dots,k-1$. Then  
$P_{\omega_1,\dots,\omega_k}=P_{\omega_1,\dots,\omega_{k-1}}\cap P_{\omega_k}$ and 
$P_{\omega_1,\dots,\omega_k}\setminus\partial P_{\omega_1,\dots,\omega_k}$ is an open subset in 
$\partial P_{\omega_1,\dots,\omega_{k-1}}$, which is an orientable
manifold, since it is~a~boundary of~an~orientable manifold. Thus, $P_{\omega_1,\dots,\omega_k}$ is orientable. 
\end{proof}

\begin{lemma}\label{lem:k=2}
Let $P$ be a simple $n$-polytope, and  let $\omega_1,\omega_2\subset[m]$ be
subsets such that $\omega_1\cap\omega_2=\varnothing$, 
and $P_{\omega_1}$, $P_{\omega_2}$ are $(n-1)$-$RHD$. 
\begin{itemize}
\item If $P_{\omega_1\sqcup\omega_2}$ is  an~$(n-1)$-$RHD$,
then $P_{\omega_1}\cap P_{\omega_2}$ is an~$(n-2)$-$RHD$. 
\item If $P_{\omega_1\sqcup\omega_2}$ is  an~$(n-1)$-$RHS$,
then $\omega_1\sqcup \omega_2=[m]$ and  $P_{\omega_1}\cap P_{\omega_2}$ is an  $(n-2)$-$RHS$.
\end{itemize}
\end{lemma}
\begin{proof}
The first item follows from Lemma \ref{lem:or}  and the~Mayer-Vietoris sequence in reduced homology 
$$
\dots\to\widetilde{H}_k(P_{\omega_1}\cap P_{\omega_2})\to 
\widetilde{H}_k(P_{\omega_1})\oplus \widetilde{H}_k(P_{\omega_2})\to
\widetilde{H}_k(P_{\omega_1\sqcup\omega_2})\to \widetilde{H}_{k-1}(P_{\omega_1}\cap P_{\omega_2})\to\dots
$$
The second item follows from these facts and the Alexander duality 
$\widetilde{H}^{k}(P_{\omega})=\widetilde{H}_{n-2-k}(P_{[m]\setminus\omega})$ in the sphere 
$\partial P\simeq S^{n-1}$, since $P_{[m]\setminus\omega}$ is homotopy equivalent to 
$\partial P\setminus P_{\omega}$. 
\end{proof}
\begin{corollary}
Let $P$ be a simple $n$-polytope, $n\leqslant 4$, and  let $\omega_1,\omega_2\subset[m]$ be
subsets such that $\omega_1\cap\omega_2=\varnothing$, 
and $P_{\omega_1}$, $P_{\omega_2}$  are  $(n-1)$-disks. 
\begin{itemize} 
\item If $P_{\omega_1\sqcup\omega_2}$ is an $(n-1)$-disk, then $P_{\omega_1}\cap P_{\omega_2}$ is an  $(n-2)$-disk.
\item If $P_{\omega_1\sqcup\omega_2}$ is an $(n-1)$-sphere, then $\omega_1\sqcup \omega_2=[m]$, and 
$P_{\omega_1}\cap P_{\omega_2}$ is an  $(n-2)$-sphere.
\end{itemize}
\end{corollary}

\section{A criterion when $N(P,\Lambda)\simeq S^4$}\label{sec:Sn4}
We will use the following result.
\begin{theorem} (Armstrong, \cite[Theorem 3]{A65}). Let $X$ be a connected
simply connected simplicial complex and let a finite group $G$ acts on $X$ 
by simplicial homeomorphisms.  Let $H$ be~a~subgroup in $G$ generated by all the elements 
$h\in G$ such that the set of fixed points $X^h$ is nonempty. 
Then $\pi_1(X/G)$ is isomorphic to $G/H$.
\end{theorem}

For any, not necessarily right-angled, simple polytope $P$ define 
{\it the right-angled Coxeter group}\index{right-angled Coxeter group} 
$$
\mathcal{G}(P)=\langle \rho_1,\dots,\rho_m\rangle/(\rho_1^2=\dots=\rho_m^2=e, \rho_i\rho_j=\rho_j\rho_i,
\text{ if }F_i\cap F_j\ne\varnothing).
$$
\begin{definition}(see \cite{V71, D83, D08})
Define a space
\begin{gather*}
W(P)=P\times \mathcal{G}(P)/\sim,\\
\text{ where }(p,g_1)\sim(q,g_2)\text{ if and only if }p=q \text{ and }g_1^{-1}g_2\in
\langle \rho_i\colon p\in F_i\rangle
 \end{gather*}
\end{definition}
\begin{proposition}\cite[Theorems 10.1 and 13.5]{D83}
We have $\pi_1(W(P))=0$.  
\end{proposition}
The proof of the following fact is~straightforward.
\begin{proposition}
The space $W(P)$ is a connected topological $n$-manifold. Moreover, it has 
a~structure of~a~simplicial complex such that $\mathcal{G}(P)$ acts on it simplicially 
(the action comes from the  action of $\mathcal{G}(P)$ on itself $g(x)=gx$).
Moreover, $W(P)/\mathcal{G}(P)\simeq P$. 
\end{proposition}
Any vector-coloring $\Lambda$ of rank $r$ defines an epimorphism $\varphi_{\Lambda}\colon \mathcal{G}(P)\to \mathbb Z_2^r$.
The proof of the following fact is~straightforward.
\begin{proposition}
We have $N(P,\Lambda)=W(P)/{{\rm Ker\,} \varphi}$.
\end{proposition}
\begin{corollary}
We have $\pi_1(N(P,\Lambda))\simeq {{\rm Ker\,} \varphi}/K$, where the subgroup $K$ is generated by 
kernels of~the~mappings $\langle g\rho_{i_1}g^{-1},\dots g\rho_{i_n}g^{-1}\rangle\simeq\mathbb Z_2^n\to\mathbb Z_2^r$, 
$g\rho_{i_s}g^{-1}\to \Lambda_{i_s}$, for all elements $g\in G$ and all vertices $F_{i_1}\cap\dots\cap F_{i_n}$ of~$P$.
\end{corollary}
\begin{proof}
We have $St_{\mathcal{G}(P)}[p,g]=\langle g\rho_ig^{-1}\colon p\in F_i\rangle$. Therefore, 
$St_{{\rm Ker\,} \varphi}[p,g]={{\rm Ker\,} \varphi}\cap \langle g\rho_ig^{-1}\colon p\in F_i\rangle$. Any such a subgroup 
lies in a subgroup corresponding to some vertex of $P$.
\end{proof}

\begin{lemma}
Let $\pi_1(N(P,\Lambda))=0$. Then different vectors among $\Lambda_1,\dots,\Lambda_m$ are linearly independent and 
form a basis in $\mathbb Z_2^r$.
\end{lemma}
\begin{proof}
If $\pi_1(N(P,\Lambda))=0$, then ${{\rm Ker\,} \varphi}$ is generated by 
kernels of~the~mappings $\langle g\rho_{i_1}g^{-1},\dots, g\rho_{i_n}g^{-1}\rangle\simeq\mathbb Z_2^n\to\mathbb Z_2^r$, 
$g\rho_{i_s}g^{-1}\to \Lambda_{i_s}$, for all elements $g\in \mathcal{G}(P)$ and all vertices $F_{i_1}\cap\dots\cap F_{i_n}$ of~$P$.

Consider the composition $\mathcal{G}(P)\to\mathbb Z_2^m\to \mathbb Z_2^r$, where the first mapping 
$\pi\colon \rho_i\to \boldsymbol{e}_i$ is~the~abelianization homomorphism, and the second mapping $\Lambda$ 
is a linear mapping defined by the condition $\boldsymbol{e}_i\to\Lambda_i$. We have 
$\pi^{-1}({\rm Ker\,}\Lambda)={\rm Ker\,}\varphi$ and 
$$
\pi\langle g\rho_{i_1}g^{-1},\dots, g\rho_{i_n}g^{-1}\rangle=\pi\langle\rho_{i_1},\dots, \rho_{i_n}\rangle=
\langle \boldsymbol{e}_{i_1},\dots, \boldsymbol{e}_{i_n}\rangle.
$$
In particular, the restriction of $\pi$ to $\langle g\rho_{i_1}g^{-1},\dots, g\rho_{i_n}g^{-1}\rangle$ 
is injective for any vertex $v$ and any $g\in\mathcal{G}(P)$. Also 
$$
\pi^{-1}({\rm Ker}\,\Lambda\left.\right|_{\langle \boldsymbol{e}_{i_1},\dots, \boldsymbol{e}_{i_n}\rangle})={\rm Ker}\,\varphi\left.\right|_{\langle g\rho_{i_1}g^{-1},\dots, g\rho_{i_n}g^{-1}\rangle}.
$$ 

Since ${{\rm Ker\,} \varphi}$ is generated by 
kernels ${\rm Ker}\,\varphi\left.\right|_{\langle g\rho_{i_1}g^{-1},\dots, g\rho_{i_n}g^{-1}\rangle}$, the group
$\pi({{\rm Ker\,} \varphi})={\rm Ker}\,\Lambda$ is~generated by~the~subgroups $\pi({\rm Ker}\,\varphi\left.\right|_{\langle g\rho_{i_1}g^{-1},\dots, g\rho_{i_n}g^{-1}\rangle})={\rm Ker}\,\Lambda\left.\right|_{\langle \boldsymbol{e}_{i_1},\dots, \boldsymbol{e}_{i_n}\rangle}$
for all vertices $F_{i_1}\cap\dots\cap F_{i_n}$ of $P$,
that is
\begin{equation}\label{eqker}
{\rm Ker}\,\Lambda=\sum\limits_{v\in P}{\rm Ker}\,\Lambda\left.\right|_{\langle \boldsymbol{e}_{i_1},\dots, \boldsymbol{e}_{i_n}\rangle}
\end{equation}
If $N(P,\Lambda)$ is a closed topological manifold, then by \cite[Theorem 5.1]{E24} at each vertex $F_{i_1}\cap\dots\cap F_{i_n}$
different vectors among $\{\Lambda_{i_1},\dots, \Lambda_{i_n}\}$ are linearly independent. In particular, 
${\rm Ker}\,\Lambda\left.\right|_{\langle \boldsymbol{e}_{i_1},\dots, \boldsymbol{e}_{i_n}\rangle}$ is generated 
by the vectors $\boldsymbol{e}_{i_a}+ \boldsymbol{e}_{i_b}$ with $\Lambda_{i_a}=\Lambda_{i_b}$.
We have  $[m]=\{1,\dots,m\}=S_1\sqcup\dots\sqcup S_t$, where 
$\Lambda_i=\Lambda_j$ (equivalently, $\boldsymbol{e}_i+ \boldsymbol{e}_j\in {\rm Ker}\,\Lambda$) 
if and only if $i$ and $j$ lie~in~the~same set~$S_k$. 
Assume that there is a linear dependence $\Lambda_{j_1}+\dots+\Lambda_{j_l}=0$, where all the vectors 
$\{\Lambda_{j_1},\dots, \Lambda_{j_l}\}$ are pairwise different, that is $j_p\in S_{i_p}$, where $i_p\ne i_q$ if $j_p\ne j_q$. 
Then $\boldsymbol{e}_{j_1}+\dots+ \boldsymbol{e}_{j_l}\in {\rm Ker}\,\Lambda$, and by (\ref{eqker})
$$
\boldsymbol{e}_{j_1}+\dots +\boldsymbol{e}_{j_l}=\sum\limits_{i=1}^t\sum\limits_{a,b\in S_i, a<b}\lambda_{a,b}(\boldsymbol{e}_a+\boldsymbol{e}_b).
$$
Since there is a direct sum $\mathbb Z_2^m=\bigoplus\limits_{i=1}^t\langle \boldsymbol{e}_j\colon j\in S_i\rangle$,
we have $\boldsymbol{e}_{j_p}=\sum\limits_{a,b\in S_{i_p}, a<b}\lambda_{a,b}(\boldsymbol{e}_a+\boldsymbol{e}_b)$. 
But on the left side the sum $x_1+\dots+x_m$ of all the coordinates of the vector is $1$, and on~the~right side it~is~$0$.
A contradiction. Thus, all different vectors among $\{\Lambda_1,\dots, \Lambda_m\}$ are linearly independent.
In particular, they form a basis in $\mathbb Z_2^r$.
\end{proof}
\begin{definition}
For a subset $\omega\subset [M]=\{1,\dots, M\}$ define  $G_{\omega}=\bigcup\limits_{i\in\omega} G_i$.
For a collection of pairwise disjoint subsets $\omega_1$, $\dots$, $\omega_k\subset[M]$ define 
$G_{\omega_1,\dots,\omega_k}=G_{\omega_1}\cap\dots\cap G_{\omega_k}$.
\end{definition}
\begin{corollary}\label{cor:k=1}
Let $N(P,\Lambda)\simeq S^n$. Then for each proper subset $\omega\subset [M]$ the set 
$G_{\omega}$ is an~$(n-1)$-$RHD$. For $\omega=[M]$ we have $G_{\omega}=\partial P\simeq S^{n-1}$.
\end{corollary}
\begin{proof}
This follows directly from Theorem \ref{HNPLth} and Proposition \ref{pr:rhs}.
\end{proof}

\begin{lemma}\label{lem:RHDS}
Let $N(P,\Lambda)\simeq S^n$. Then for 
each collection of pairwise disjoint proper subsets $\omega_1$, $\dots$, $\omega_k\subset[M]$ 
the set $G_{\omega_1,\dots,\omega_k}$ is 
\begin{itemize}
\item an~$(n-k)$-$RHD$, if $\omega_1\sqcup\dots\sqcup\omega_k\ne[M]$;
\item an~$(n-k)$-$RHS$, if $\omega_1\sqcup\dots\sqcup\omega_k=[M]$;
\end{itemize} 
\end{lemma}
\begin{proof}
We will prove this by induction on $k$. For $k\leqslant 2$ this follows from Lemma \ref{lem:k=2} and~Corollary \ref{cor:k=1}.
Let the lemma hold for $1$, $\dots$, $k-1$. We have 
$G_{\omega_1,\dots,\omega_k}=
G_{\omega_1,\dots,\omega_{k-2},\omega_{k-1}}\cap G_{\omega_1,\dots,\omega_{k-2},\omega_k}$,
where $$
G_{\omega_1,\dots,\omega_{k-2},\omega_{k-1}}\cup G_{\omega_1,\dots,\omega_{k-2},\omega_k}=
G_{\omega_1,\dots,\omega_{k-2},\omega_{k-1}\sqcup \omega_k}.
$$ 
Now the proof follows from the Mayer-Vietoris sequence and Lemma \ref{lem:or}. 
\end{proof}
\begin{corollary}\label{cor:NPSn}
Let $N(P,\Lambda)\simeq S^n$. Then $M\leqslant n+1$ and each $k$-face of the complex $C(P,\Lambda)$
is either a $k$-$RHD$ or a $k$-$RHS$.
\end{corollary}
\begin{proof}
Indeed, $G_1\cap\dots\cap G_{M-1}$ is an~$(n-M+1)$-$RHD$. In particular,
it is nonempty. But 
the~intersection of any $n+1$ facets of $P$ is empty, hence $M-1\leqslant n$.  Each $k$-face of $C(P,\Lambda)$
is~a~connected component of $G_{i_1}\cap\dots\cap G_{i_{n-k}}$, hence by Lemma~\ref{lem:RHDS}
it~is~a~$k$-$RHD$ if $n-k<M$, and a~$k$-$RHS$ if $n-k=M$. 
\end{proof}
\begin{corollary}\label{cor:34nM}
Let $n\leqslant 4$. If $N(P,\Lambda)\simeq S^n$, then $M\leqslant n+1$
and for each collection of~pairwise disjoint proper subsets $\omega_1$, $\dots$, $\omega_k\subset[M]$ 
the set $G_{\omega_1,\dots,\omega_k}$ is 
\begin{itemize}
\item an~$(n-k)$-disk, if $\omega_1\sqcup\dots\sqcup\omega_k\ne[M]$;
\item an~$(n-k)$-sphere, if $\omega_1\sqcup\dots\sqcup\omega_k=[M]$.
\end{itemize}
\end{corollary}
\begin{proof}
This follows from Lemmas~\ref{lem:RHDS} and~\ref{lem:rhd4}
and the fact that an~$n$-$RHS$ is a topological $n$-sphere for $n\leqslant 2$ and an~$n$-$RHD$ is a topological $2$-disk for 
$n\leqslant 2$.
\end{proof}
\begin{lemma}\label{lem:34nM}
Let $n\leqslant 4$. If for a complex $\mathcal{C}(P,\Lambda)$ we have $M\leqslant n+1$
and for each collection of pairwise disjoint proper subsets $\omega_1$, $\dots$, $\omega_k\subset[M]$ 
the set $G_{\omega_1,\dots,\omega_k}$ is 
\begin{itemize}
\item an~$(n-k)$-disk, if $\omega_1\sqcup\dots\sqcup\omega_k\ne[M]$;
\item an~$(n-k)$-sphere, if $\omega_1\sqcup\dots\sqcup\omega_k=[M]$,
\end{itemize}
then $\mathcal{C}(P,\Lambda)$ is equivalent to $\mathcal{C}(n,M)$.
\end{lemma}
\begin{proof}
We will show that complexes $\mathcal{C}(P,\Lambda)$ and $\mathcal{C}(n,M)$ have the same  
combinatorial structure. Then an equivalence can be constructed inductively starting from the skeleton
of minimal dimension, which is either the set of $(n+1)$ vertices for $M=n+1$, or $(n-M)$-sphere for $M<n+1$. 
Each face of greater dimension is a disk such that on its boundary a homeomorphism is already constructed. 
We extend the homeomorphism using the fact that the disk is the cone on its boundary.

For $n=2$ this is trivial. 

Let $n=3$. If $M=1$, then $\mathcal{C}(P,\Lambda)\simeq \mathcal{C}(n,1)$.  If $M=2$,
then $G_1$ and $G_2$ are $2$-disks such that $\partial G_1=\partial G_2=G_1\cap G_2$ is~a~circle. Hence,
$\mathcal{C}(P,\Lambda)\simeq \mathcal{C}(n,2)$. If $M=3$, then $G_1$, $G_2$ and $G_3$ are $2$-disks 
such that $G_i\cap G_j$ is a simple path for any $i\ne j$, and all these three paths have common ends,
which form the $0$-sphere $G_1\cap G_2\cap G_3$. Thus, the $1$-skeleton of $\mathcal{C}(P,\Lambda)$ is~a~theta-graph,
and $\mathcal{C}(P,\Lambda)\simeq \mathcal{C}(n,3)$. If $M=4$, then each $G_i$ is a disk, the~intersection 
$G_i\cap G_j$ of each two such disks is a simple path, the intersection of any three such disks is a point, and the intersection
of all the four disks is empty. Then the $1$-skeleton of $\mathcal{C}(P,\Lambda)$ is a $K_4$-graph, and  
$\mathcal{C}(P,\Lambda)\simeq \mathcal{C}(n,3)$.

Let $n=4$. If $M=1$, then $\mathcal{C}(P,\Lambda)\simeq \mathcal{C}(n,1)$. If $M=2$,
then $G_1$ and $G_2$ are $3$-disks such that $\partial G_1=\partial G_2=G_1\cap G_2$ is~a~$2$-sphere. 
Hence, $\mathcal{C}(P,\Lambda)\simeq \mathcal{C}(n,2)$. If $M=3$, then $G_1$, $G_2$ and $G_3$ are $3$-disks,
such that $G_i\cap G_j$ is a $2$-disk for $i<j$, and $G_1\cap G_2\cap G_3$ is a circle. We have $G_1\cup G_2$
is a $3$-disk with a circle $G_1\cap G_2\cap G_3=\partial (G_1\cap G_2)$ on its boundary, and the complement
to this disk is the interior of $G_3$. Hence, 
$\mathcal{C}(P,\Lambda)\simeq \mathcal{C}(n,3)$. If $M=4$, then  $G_1$, $G_2$, $G_3$, $G_4$ are $3$-disks,
$G_i\cap G_j$ is a $2$-disk for any $i< j$,  $G_i\cap G_j\cap G_k$ is a simple path for any $i<j<k$, and all these
four paths have the~common ends, which form the $0$-sphere $G_1\cap G_2\cap G_3\cap G_4$. Then $G_1\cup G_2\cup G_3$
is a $3$-disk with the theta-graph on the boundary formed by
paths $G_i\cap G_j\cap G_4$, $i<j<4$. There is the path $G_1\cap G_2\cap G_3$ inside this disk
connecting the vertices of~the~theta-graph and each $G_i\cap G_j$ is a $2$-disk bounded by the paths
$G_1\cap G_2\cap G_3$ and $G_i\cap G_j\cap G_4$. The complement to $G_1\cup G_2\cup G_3$ in $\partial P$
is the interior of $G_4$. Thus, $\mathcal{C}(P,\Lambda)\simeq \mathcal{C}(n,4)$. If $M=5$, then each $G_i$
is a $3$-disk, each intersection $G_i\cap G_j$, $i<j$, is a $2$-disk, each intersection $G_i\cap G_j\cap G_k$,
$i<j<k$, is a simple path, each intersection $G_i\cap G_j\cap G_k\cap G_l$, $i<j<k<l$, is a point,
and $G_1\cap G_2\cap G_3\cap G_4\cap G_5=\varnothing$. Then    $G_1\cup G_2\cup G_3\cup G_4$
is a~$3$-disk with a~$K_4$-graph on~its~boundary formed by vertices $G_i\cap G_j\cap G_k\cap G_5$
and edges $G_i\cap G_j\cap G_5$. There is a point $G_1\cap G_2\cap G_3\cap G_4$ inside this disk
such that the face structure in the disk $G_1\cup G_2\cup G_3\cup G_4$ is~a~cone over the~face structure on its boundary 
with this apex. Also the complement to this disk is the interior of~$G_5$. Thus, 
$\mathcal{C}(P,\Lambda)\simeq \mathcal{C}(n,5)\simeq\partial \Delta^4$.

\end{proof}
As a corollary we obtain the following theorem.
\begin{theorem}\label{th:n34}
Let $\Lambda$ be a vector-coloring of~rank~$r$ of~a~simple polytope $P$ of dimension $n\leqslant 4$. 
Then $N(P,\Lambda)\simeq S^n$ if and only if $\mathcal{C}(P,\Lambda)\simeq \mathcal{C}(n,r)$.
\end{theorem}
\begin{proof}
The if part follows from \cite[Construction 5.8]{E24}. The only if part follows from Corollary~\ref{cor:34nM} and 
Lemma~\ref{lem:34nM}.
\end{proof}

\section{Hyperelliptic manifolds $N(P,\Lambda)$}\label{sec:hyp}
Theorem \ref{th:n34} implies the following result (see \cite[Construction 8.6]{E24} and \cite[Corollary 10.8]{E24}).
\begin{corollary}
Let  $\Lambda$ be a vector-coloring of~rank $r$ of a simple $n$-polytope $P$, $n\leqslant 4$.
Then an involution $\tau\in\mathbb Z_2^r$ is hyperelliptic (that is $N(P,\lambda)\simeq S^n$)  if and only if 
$\tau$ is special (that is $\mathcal{C}(P,\Lambda_\tau)\simeq \mathcal{C}(n,r-1)$, where $\Lambda_\tau$ 
is the composition of $\Lambda$ and the projection $\mathbb Z_2^r\to\mathbb Z_2^r/\langle\tau\rangle\simeq \mathbb Z_2^{r-1}$).
\end{corollary}
\subsection{Hamiltonian subcomplexes}\label{ssec:Hams}
\begin{construction}\label{con:subcom}
For two colorings $c_1\colon\{F_1,\dots, F_m\}\to [l_1]$ and \linebreak $c_2\colon\{F_1,\dots, F_m\}\to [l_2]$
we call $\mathcal{C}(P,c_2)$ a~{\it subcomplex} in $\mathcal{C}(P,c_1)$ 
(and write $\mathcal{C}(P,c_2)\subset \mathcal{C}(P,c_1)$), if there is a~mapping 
$\pi\colon [l_1]\to[l_2]$ such that $c_2=\pi\circ c_1$. Each facet of~$\mathcal{C}(P,c_2)$ is a union of facets 
of~$\mathcal{C}(P,c_1)$. More generally, for any $q\geqslant 0$ 
each $q$-face of~$\mathcal{C}(P,c_2)$ is a~union of $q$-faces of~$\mathcal{C}(P,c_1)$.
\end{construction}
\begin{definition}\label{def:Ham}
We call a subcomplex  $\mathcal{C}(P,c_2)\subset \mathcal{C}(P,c_1)$ {\it Hamiltonian}, if each $q$-skeleton 
of~$\mathcal{C}(P,c_1)$ lies in the~$(q+1)$-skeleton of $\mathcal{C}(P,c_2)$.
\end{definition}

\begin{lemma}\label{lem:Hamsc}
A subcomplex $\mathcal{C}(P,c_2)\subset \mathcal{C}(P,c_1)$ is Hamiltonian if and only if
one of the following equivalent conditions holds:
\begin{enumerate}
\item For each nonempty intersection $G_{i_1}\cap \dots\cap G_{i_k}\ne\varnothing$ of different facets 
of $\mathcal{C}(P,c_1)$ at least $(k-1)$ of these facets lie in different facets of $\mathcal{C}(P,c_2)$.
\item For each nonempty intersection $G_{i_1}\cap \dots\cap G_{i_k}\ne\varnothing$ of $k=3$ or $k=4$ 
different facets of $\mathcal{C}(P,c_1)$ at least $(k-1)$ of these facets lie in different facets of $\mathcal{C}(P,c_2)$.
\item The following two conditions hold:
\begin{itemize}
\item $G_{i_1}\cap G_{i_2}\cap G_{i_3}=\varnothing$, if $c_2(G_{i_1})=c_2(G_{i_2})=c_2(G_{i_3})$ and 
$i_1<i_2<i_3$,  
\item $G_{i_1}\cap G_{i_2}\cap G_{i_3}\cap G_{i_4}=\varnothing$, if 
$c_2(G_{i_1})=c_2(G_{i_2})\ne c_2(G_{i_3})=c_2(G_{i_4})$ and 
$i_1\ne i_2$, $i_3\ne i_4$;
\end{itemize} 
\item Any $(n-4)$-face of $\mathcal{C}(P,c_1)$ lies in an $(n-3)$-face of $\mathcal{C}(P,c_2)$,
and any closed manifold among $(n-3)$-faces of $\mathcal{C}(P,c_1)$ lies in an $(n-2)$-face of~$\mathcal{C}(P,c_2)$. 
\end{enumerate}
\end{lemma}
\begin{proof}
By Lemma \ref{lem:Mc} the facets of $C(P,c)$ intersect locally as coordinate hyperplanes in $\mathbb R^n$ (moreover,
other facets do not intersect this local neighbourhood). By definition 
if $\mathcal{C}(P,c_2)\subset \mathcal{C}(P,c_1)$ is Hamiltonian, then the set $G_{i_1}\cap \dots\cap G_{i_k}$ 
lies in the $(n-k-1)$-skeleton of $\mathcal{C}(P,c_2)$. Take any point $p\in G_{i_1}\cap \dots\cap G_{i_k}$
such that $p\not\in G_{j_l}$ for each facet $G_{i_l}$ of $\mathcal{C}(P,c_1)$ different from $G_{i_1},\dots, G_{i_k}$ 
(such a point exists by~the~above argument). Then 
$p\in \widetilde{G}_{j_1}\cap \dots\cap \widetilde{G}_{j_{k-1}}$, where 
$\widetilde{G}_{j_1}, \dots, \widetilde{G}_{j_{k-1}}$ are different facets of $\mathcal{C}(P,c_2)$. We have
$ \widetilde{G}_{j_p}=\bigcup\limits_{G_q\subset \widetilde{G}_{j_p}}G_q$ and 
$$
\widetilde{G}_{j_1}\cap \dots\cap \widetilde{G}_{j_{k-1}}=\left(\bigcup\limits_{G_{q_1}\subset \widetilde{G}_{j_1}}G_{q_1}\right)
\cap\dots\cap\left(\bigcup\limits_{G_{q_{k-1}}\subset \widetilde{G}_{j_{k-1}}}G_{q_{k-1}}\right)=
\bigcup\limits_{q_1,\dots,q_{k-1}}G_{q_1}\cap\dots\cap G_{q_{k-1}}
$$
Then $p\in G_{q_1}\cap\dots\cap G_{q_{k-1}} $ for some $q_1,\dots, q_{k-1}$. Therefore, by the choice 
of $p$ we have $\{q_1,\dots,q_{k-1}\}\subset \{i_1,\dots,i_k\}$, where $G_{q_p}\subset  \widetilde{G}_{j_p}$.  This implies 
item (1). On the other hand, if item (1) holds, then by definition the subcomplex $\mathcal{C}(P,c_2)\subset \mathcal{C}(P,c_1)$
is Hamiltonian.

Item (1) implies item (2). Item (3) is a reformulation of item (2), so these items are equivalent.

Let item (2) hold.  Consider a nonempty intersection  $G_{i_1}\cap \dots\cap G_{i_k}\ne\varnothing$ of different facets 
of $\mathcal{C}(P,c_1)$. For $k\leqslant 4$ item (1) holds. Assume that $k\geqslant 5$ and item (1) does not hold. 
Then $|\{c_2(G_{i_1}),\dots, c_2(G_{i_k})\}|\leqslant k-2$ and either $c_2(G_{i_{p_1}})=c_2(G_{i_{p_2}})=c_2(G_{i_{p_2}})$
for $p_1<p_2<p_3$, or~$c_2(G_{i_{p_1}})=c_2(G_{i_{p_2}})
\ne c_2(G_{i_{p_3}})=c_2(G_{i_{p_4}})$
for $p_1\ne p_2$ and $p_3\ne p_4$. This contradicts item~(3). 

Item (2) implies item (4). If item (4) holds, then the argument in the beginning of the proof
implies that item (3) hold for $k=4$ and closed manifolds for $k=3$. If $G_{i_1}\cap G_{i_2}\cap G_{i_3}$
is not a~closed manifold, then there is a facet $G_{i_4}$ of $\mathcal{C}(P,c_1)$ such that 
$G_{i_1}\cap G_{i_2}\cap G_{i_3}\cap G_{i_4}\ne\varnothing$. Then by the above argument it is 
impossible that $G_{i_1}$, $G_{i_2}$ and $G_{i_3}$ lie in the same facet of $\mathcal{C}(P,c_2)$. 
Hence, item (3) holds. 
\end{proof}

Lemma \ref{lem:Hamsc} leads to the following definition.
\begin{definition}\label{def:n-2f}
For a subcomplex $\mathcal{C}(P,c_2)\subset \mathcal{C}(P,c_1)$ let us call
each connected component of~a~nonempty intersection $G_{i_1}\cap G_{i_2}$, where $G_{i_1}, G_{i_2}$ are facets 
of~$\mathcal{C}(P,c_1)$ and $c_2(G_{i_2})=c_2(G_{i_2})$, a~{\it defining} $(n-2)$-face.
\end{definition}
\begin{proposition}\label{cor:n-2f}
A subcomplex $\mathcal{C}(P,c_2)\subset \mathcal{C}(P,c_1)$ is Hamiltonian if and only if any two 
defining $(n-2)$-faces are disjoint.
\end{proposition}
\begin{proof}
Let any two defining $(n-2)$-faces be disjoint. If 
$c_2(G_{i_1})=c_2(G_{i_2})=c_2(G_{i_3})$ and $i_1<i_2<i_3$, then either
one of the intersections $G_{i_1}\cap G_{i_2}$ and $G_{i_1}\cap G_{i_3}$ is empty, or
each intersection is a disjoint union of defining faces and any two faces from different
unions are also disjoint. Then
$$
G_{i_1}\cap G_{i_2}\cap G_{i_3}=(G_{i_1}\cap G_{i_2})\cap (G_{i_1}\cap G_{i_3})=\varnothing.
$$
If $c_2(G_{i_1})=c_2(G_{i_2})\ne c_2(G_{i_3})=c_2(G_{i_4})$ and $i_1\ne i_2$, $i_3\ne i_4$,
then either 
one of the intersections $G_{i_1}\cap G_{i_2}$ and $G_{i_3}\cap G_{i_4}$ is empty, or
each intersection is a disjoint union of defining faces and any two faces from different
unions are also disjoint. Then 
$$
G_{i_1}\cap G_{i_2}\cap G_{i_3}\cap G_{i_4}=(G_{i_1}\cap G_{i_2})\cap (G_{i_3}\cap G_{i_4})=\varnothing
$$
Then the subcomplex is Hamiltonian by Lemma~\ref{lem:Hamsc}(3).

Now let the subcomplex be Hamiltonian. Consider two defining faces. If they are connected components
of~the~same intersection $G_{i_1}\cap G_{i_2}$, then they are disjoint by definition. If they lie in different
intersections $G_{i_1}\cap G_{i_2}$ and $G_{i_3}\cap G_{i_4}$, then there are two possibilities. First possibility is 
$c(G_{i_1})=c(G_{i_2})=c(G_{i_3})=c(G_{i_4})$.  Then without loss of generality we
can assume that $i_3\notin\{i_1,i_2\}$. Hence,  $G_{i_1}\cap G_{i_2}\cap G_{i_3}=\varnothing$ by Lemma~\ref{lem:Hamsc}(3)
and $(G_{i_1}\cap G_{i_2})\cap (G_{i_3}\cap G_{i_4})=\varnothing$. The second possibility is 
$c(G_{i_1})=c(G_{i_2})\ne c(G_{i_3})=c(G_{i_4})$. Then $(G_{i_1}\cap G_{i_2})\cap (G_{i_3}\cap G_{i_4})=\varnothing$
by  Lemma~\ref{lem:Hamsc}(3). This finishes the proof.
\end{proof}

\begin{proposition}\label{prop:critdef}
A set $\mathcal{S}=\{M_1^{n-2},\dots, M_q^{n-2}\}$ of pairwise disjoint $(n-2)$-faces of~a~complex $\mathcal{C}(P,c_1)$ is a set
of defining faces of some Hamiltonian subcomplex $\mathcal{C}(P,c_2)\subset \mathcal{C}(P,c_1)$ if~and only~if~for each 
two facets $G_i\ne G_j$ of~$\mathcal{C}(P,c_1)$ admitting  a~sequence 
of~facets $G_i=G_{i_1}$, $G_{i_2}$, $\dots$, $G_{i_p}=G_j$ such that for all $s$ the
intersection $G_{i_s}\cap G_{i_{s+1}}$ contains a face from $\mathcal{S}$ either $G_i\cap G_j=\varnothing$ or 
each connected component of $G_i\cap G_j$ belongs to $\mathcal{S}$. 
\end{proposition}
\begin{proof}
If $\mathcal{S}$ satisfies the above conditions then one can define a coloring $c_2$ by the 
rule $c_2(G_i)=c_2(G_j)$ if and only if  $G_i$ and $G_j$ admit a sequence of the above type. Then
each connected component of the intersection of two facets of the same color belongs to $\mathcal{S}$ and 
this is indeed the set of defining facets. Since any two of them are disjoint, the subcomplex is Hamiltonian.

On the other hand, if $\mathcal{S}$ is the set of defining faces of a Hamiltonian subcomplex, 
then the facets $G_i$ and $G_j$ admitting the 
above sequence necessarily have the same color. Hence, if~$G_i\cap G_j\ne\varnothing$, 
then each connected component of $G_i\cap G_j$ belongs to $\mathcal{S}$.
\end{proof}

\subsection{Bipartite Hamiltonian subcomplexes and orbit spaces of involutions}\label{ssec:bip}
\begin{definition}\label{def:adjg}
For a Hamiltonian subcomplex $\mathcal{C}(P,c_2)\subset\mathcal{C}(P,c_1)$ 
and a facet $\widetilde{G}\in \mathcal{C}(P,c_2)$ define an {\it adjacency graph $\Gamma(\widetilde{G})$}. 
Its vertices are facets $G_i$ of $\mathcal{C}(P,c_1)$ lying in $\widetilde{G}$. Its edges bijectively correspond
to connected components of nonempty intersections $G_i\cap G_j\ne \varnothing$ of~such facets,
that is to defining $(n-2)$-faces lying in $\widetilde{G}$. This graph is connected and may have multiple edges
corresponding to different connected components of $G_i\cap G_j\ne \varnothing$.
Define an~{\it adjacency graph} 
$$
\Gamma(P,c_1,c_2)=\bigsqcup\limits_{\widetilde{G} \text{ -- a facet of }\mathcal{C}(P,c_2)} \Gamma(\widetilde{G}).
$$
\end{definition}
\begin{definition}\label{def:bip}
We call a Hamiltonian subcomplex $\mathcal{C}(P,c_2)\subset\mathcal{C}(P,c_1)$ {\it bipartite},
if the adjacency graph  $\Gamma(P,c_1,c_2)$ is bipartite, that is 
there is a coloring \linebreak $\chi\colon \{G_1,\dots, G_{M_1}\}\to\{0,1\}$ of~facets of~$\mathcal{C}(P,c_1)$ in two colors $0$ and $1$
such that if $G_i$ and $G_j$ lie in the same facet of $\mathcal{C}(P,c_2)$ and $G_i\cap G_j\ne\varnothing$,
then  $\chi(G_i)\ne \chi(G_j)$. We will call such colorings {\it nice}.
\end{definition}
\begin{remark}\label{rem:2^M}
If  $\mathcal{C}(P,c_2)\subset\mathcal{C}(P,c_1)$ is a proper subcomplex and $\mathcal{C}(P,c_2)$ has $M_2$ facets,
then there are $2^{M_2}$ nice colorings $\chi$.
\end{remark}

\begin{construction}[(The vector-coloring induced {\bf from} a bipartite \linebreak Hamiltonian subcomplex)]\label{con:indfrom}
Let~$\Lambda\colon \{F_1,\dots, F_m\}\to \mathbb Z_2^r$ be a vector-coloring of rank $r$ of~a~simple $n$-polytope $P$.
Let~$\mathcal{C}(P,\Lambda)\subset\mathcal{C}(P,c)$ be a proper bipartite Hamiltonian subcomplex
and $\chi\colon \{G_1,\dots, G_{M_1}\}\to\{0,1\}$ be one of~its~nice colorings. Define the vector-coloring
$\Lambda_\chi\colon \{F_1,\dots, F_m\}\to \mathbb Z_2^r\times \mathbb Z_2\simeq \mathbb Z_2^{r+1}$ 
of~rank~$r+1$ as~$\Lambda_\chi(F_i)=(\Lambda(F_i),\chi(F_i))$. We call $\Lambda_\chi$, as~well as~any vector-coloring 
obtained from it~by~a~linear change of coordinates in $\mathbb Z_2^{r+1}$, a~vector-coloring {\it induced from 
a~bipartite Hamiltonian subcomplex}. By definition $\mathcal{C}(P,c)=\mathcal{C}(P,\Lambda_\chi)$.
\end{construction}
\begin{remark}
It can be shown that in general different nice colorings $\chi_1$ and $\chi_2$ may produce vector colorings $\Lambda_{\chi_1}$
and $\Lambda_{\chi_2}$ that can not be connected by~a~linear change of coordinates in~$\mathbb Z_2^{r+1}$. 

\end{remark}
\begin{example}[(A vector-coloring induced {\bf by} a bipartite Hamiltonian subcomplex)]\label{ex:vci}
For any complex $C=\mathcal{C}(P,c)$ with $M$ facets $G_1$, $\dots$, $G_M$ 
there is a canonical vector-coloring $\Lambda_C$ of rank $M$ of the polytope $P$ 
defined as $\Lambda_C(F_i)=\boldsymbol{e}_j$, where 
$F_i\subset G_j$ and $\boldsymbol{e}_1$, $\dots$, $\boldsymbol{e}_M$ is~the~standard basis in~$\mathbb Z_2^M$.
(Note that different colorings $c$ may produce the~same complex $C=\mathcal{C}(P,c)$. Therefore, $\Lambda_C$ is 
defined by the coloring $c$ only if~for~each color the~union of~facets of~this color is~connected.) Then if  
$C=\mathcal{C}(P,c_2)\subset\mathcal{C}(P,c_1)$ is~a~proper bipartite Hamiltonian subcomplex, there 
is~a~canonical vector-coloring $\Lambda_{c_2, c_1}$ {\it induced by $C$}. It~is~induced from the canonical 
vector-coloring $\Lambda_C$ described above. Moreover, in this case for each $i=1,\dots, M$
there is a~linear isomorphism defined on basis as
$$
(\boldsymbol{e}_1,\dots, \boldsymbol{e}_i,\dots, \boldsymbol{e}_M, \boldsymbol{e}_{M+1})\to
(\boldsymbol{e}_1,\dots, \boldsymbol{e}_i+\boldsymbol{e}_{M+1},\dots, \boldsymbol{e}_M, \boldsymbol{e}_{M+1}) 
$$ 
exchanging $\boldsymbol{e}_i$ and $\boldsymbol{e}_i+\boldsymbol{e}_{M+1}$ and 
fixing $\boldsymbol{e}_j$ and $\boldsymbol{e}_j+\boldsymbol{e}_{M+1}$ for all $j\ne i$. Therefore, $\Lambda_{c_2, c_1}$
does not depend on $\chi$ up to a change of coordinates in $\mathbb Z_2^{M+1}$ (see also 
\cite[Construction 8.6]{E24}).

\end{example}

For any vector-coloring  $\Lambda$ of rank $r$ each 
nonzero involution $\tau\in \mathbb Z_2^r\simeq\mathbb Z_2^m/H(\Lambda)$
corresponds to a projection $\Pi_{\tau}\colon \mathbb Z_2^r\to\mathbb Z_2^r/\langle \tau\rangle\simeq \mathbb Z_2^{r-1}$
and a vector-coloring $\Lambda_\tau=\Lambda\circ\Pi_\tau$ such that 
$N(P,\Lambda)/\langle\tau\rangle\simeq N(P,\Lambda_\tau)$. 

\begin{proposition} \label{propo:taubpH}
Let $\Lambda$ be a~vector-coloring of rank $r$ of~a~simple $n$-polytope $P$ such that $N(P,\Lambda)$
is~a~closed topological manifold. Then nonzero involutions 
$\tau\in \mathbb Z_2^r\setminus\{0\}$ 
are in~bijection with proper bipartite Hamiltonian subcomplexes $\mathcal{C}(P,c)\subset\mathcal{C}(P,\Lambda)$
such that one of the following equivalent conditions~holds:
\begin{enumerate}
\item $\Lambda$ is induced from some vector-coloring of~$\mathcal{C}(P,c)$;
\item there is $\tau\in \mathbb Z_2^r\setminus\{0\}$ such that 
for any nonempty intersection of two different facets $G_i$ and $G_j$ of~$\mathcal{C}(P,\Lambda)$
we have $\Lambda_i+\Lambda_j=\tau$ if and only if $c(G_i)=c(G_j)$.
\end{enumerate}
\end{proposition}
\begin{proof}
First let us prove that for a~proper bipartite Hamiltonian subcomplex 
$\mathcal{C}(P,c)\subset\mathcal{C}(P,\Lambda)$ conditions (1) and (2) are equivalent. 
Indeed, by Construction~\ref{con:indfrom} (1) implies (2).
On~the~other hand, if (2) holds, then the vector-coloring $\Lambda_\tau$ is constant on facets of~$\mathcal{C}(P,c)$
and differs for~its~adjacent facets. Thus, $\mathcal{C}(P,c)=\mathcal{C}(P,\Lambda_\tau)$.
Using~the~change of~coordinates corresponding to an isomorphism 
$\mathbb Z_2^r/\langle\tau\rangle\oplus\langle\tau\rangle\simeq \mathbb Z_2^r$ we see that 
$\Lambda$ is induced from~$\Lambda_\tau$. Thus, (1) holds.

Now consider a~nonzero involution $\tau\in \mathbb Z_2^r$. The~subcomplex  
$\mathcal{C}(P,\Lambda_\tau)\subset \mathcal{C}(P,\Lambda)$
is Hamiltonian by \cite[Proposition 5.12]{E24}. Moreover, it is bipartite. Indeed, for any two 
facets $G_i$, $G_j$ of~$\mathcal{C}(P,\Lambda)$ lying in the same facet of $\mathcal{C}(P,\Lambda_\tau)$
we have $\Pi_\tau(\Lambda_i)=\Pi_\tau(\Lambda_j)$. Then either $\Lambda_i=\Lambda_j$, or $\Lambda_i=\Lambda_j+\tau$.
Moreover, if $G_i\cap G_j\ne\varnothing$, then $\Lambda_i=\Lambda_j+\tau$, since these vectors are different.
Also condition (2) holds by definition of~$\Lambda_\tau$.

On the other hand, for a~proper bipartite Hamiltonian subcomplex $\mathcal{C}(P,c)$ the condition (2) uniquely 
defines the involution $\tau$ such that $\mathcal{C}(P,c)=\mathcal{C}(P,\Lambda_\tau)$.
\end{proof}
\begin{proposition} \label{prop:Plman}
Let $\Lambda$ be a~vector-coloring of rank $r$ of~a~simple $n$-polytope $P$ such that $N(P,\Lambda)$
is~a~closed topological manifold. For an involution $\tau\in \mathbb Z_2^r\setminus\{0\}$ 
the space $N(P,\Lambda)/\langle \tau\rangle$ is a~closed topological manifold if and only if 
\begin{gather*}
\tau\notin 
\{\Lambda_{i_1}+\dots+\Lambda_{i_k}\colon k\ne 2, i_1<\dots<i_k, G_{i_1}\cap \dots\cap G_{i_k}\ne\varnothing\}=\\ 
\{\Lambda_{j_1}+\dots+\Lambda_{j_k}\colon k\ne 2, j_1<\dots<j_k, F_{j_1}\cap \dots\cap F_{j_k}\ne\varnothing, 
\Lambda_{j_a}\ne \Lambda_{j_b} \text{ for $a\ne b$}\}.
\end{gather*}
\end{proposition}
\begin{proof}
By \cite[Theorem 5.1]{E24}  $N(P,\Lambda)$ is~a~closed topological manifold of and only if for any
vertex $F_{j_1}\cap\dots\cap F_{j_n}$ of $P$ different vectors among $\{\Lambda_{j_1},\dots,\Lambda_{j_n}\}$
are linearly independent. This is equivalent to the fact that for any collection of indices $i_1<\dots<i_k$, $k\geqslant 1$, 
such that $G_{i_1}\cap \dots\cap G_{i_k}\ne\varnothing$ the vectors $\{\Lambda_{i_1},\dots, \Lambda_{i_k}\}$
are linearly independent and to the fact that for any collection of indices $j_1<\dots<j_k$, $k\geqslant 1$, 
such that $F_{j_1}\cap \dots\cap F_{j_k}\ne\varnothing$  and~$\Lambda_{j_a}\ne \Lambda_{j_b} \text{ for $a\ne b$}$ 
the vectors $\{\Lambda_{j_1},\dots, \Lambda_{j_k}\}$
are linearly independent. 

Let us denote by $[\Lambda_j]$ the image of $\Lambda_j$ under the projection 
$\mathbb Z_2^r\to\mathbb Z_2^r/\langle\tau\rangle$. Assume that for a~collection of indices $j_1<\dots<j_k$, $k\geqslant 1$, 
such that $F_{j_1}\cap \dots\cap F_{j_k}\ne\varnothing$  we have $[\Lambda_{j_a}]\ne [\Lambda_{j_b}] \text{ for $a\ne b$}$. 
Then $\Lambda_{j_a}+\Lambda_{j_b}\notin \{0,\tau\}$. For $\{[\Lambda_{j_1}],\dots, [\Lambda_{j_k}]\}$ to be linearly 
independent it is necessary and sufficient that $\Lambda_{j_{q_1}}+\dots+\Lambda_{j_{q_s}}\notin\{0,\tau\}$
for any nonempty subset $\{q_1,\dots, q_s\}\subset\{j_1,\dots,j_k\}$. Since $N(P,\Lambda)$ is~a~closed manifold,
$\Lambda_{j_{q_1}}+\dots+\Lambda_{j_{q_s}}\ne 0$. This finishes the~proof.
\end{proof}

\begin{proposition}\label{prop:indman}
Let $\Lambda$ be a~vector-coloring of rank $r$ of~a~simple $n$-polytope $P$ such that $N(P,\Lambda)$
is~a~closed topological manifold. Let~$\mathcal{C}(P,\Lambda)\subset\mathcal{C}(P,c)$ be 
a proper bipartite Hamiltonian subcomplex. Then for any nice coloring $\chi$ the space 
$N(P,\Lambda_{\chi})$ is~a~closed topological manifold.
\end{proposition}
\begin{proof}
Indeed, by Lemma \ref{lem:Hamsc}(1) if $G_{i_1}\cap \dots\cap G_{i_k}\ne\varnothing$, then at least 
$k-1$ of these facets lie in~different facets of $\mathcal{C}(P,\Lambda)$. Then either all of them
lie in different facets, or $k-1$ lie in different facets and two of them lie in the same facet. By definition
of $\Lambda_{\chi}$ the vectors $\Lambda_{\chi}(G_{i_1}), \dots, \Lambda_{\chi}(G_{i_k})$ are linearly independent.
\end{proof}

\subsection{Bipartite Hamiltonian subcomplexes and the four color theorem}\label{ssec:4c}
Orientable small covers $N(P,\Lambda)$ over $3$-polytopes correspond to colorings of $P$ in at most four colors. 
In this case the image of $\Lambda$ consists either of three linearly independent vectors 
$\boldsymbol{v}_1, \boldsymbol{v}_2$, $\boldsymbol{v}_3$, or four vectors
$\boldsymbol{v}_1$, $\boldsymbol{v}_2$, $\boldsymbol{v}_3$, $\boldsymbol{v}_4$ such that each three of them are linearly independent and $\boldsymbol{v}_1+\boldsymbol{v}_2+\boldsymbol{v}_3+\boldsymbol{v}_4=0$.
In what follows $N(P,\Lambda)$ is a small cover of this form.
\begin{lemma}
For $\tau\in \mathbb Z_2^3\setminus\{0\}$ the space $N(P,\Lambda)/\langle\tau\rangle=N(P,\Lambda_\tau)$
is a closed $3$-manifold if~and~only~if~$\tau\in\{\boldsymbol{v}_1+\boldsymbol{v}_2,\boldsymbol{v}_1+\boldsymbol{v}_3,
\boldsymbol{v}_2+\boldsymbol{v}_3\}$. Moreover, each manifold $N(P,\Lambda_\tau)$ is hyperelliptic 
with the~hyperelliptic involution $\mu=[\boldsymbol{v}_1+\boldsymbol{v}_3]$, $[\boldsymbol{v}_2+\boldsymbol{v}_3]$, 
$[\boldsymbol{v}_1+\boldsymbol{v}_2]$ respectively.
\end{lemma}
\begin{proof}
The first statement is a corollary of Proposition \ref{prop:Plman}. For the manifold $N(P,\Lambda_\tau)$
we~have $N(P,\Lambda_\tau)/\langle\mu\rangle=N(P,(\Lambda_\tau)_\mu)$, where $(\Lambda_\tau)_\mu$ is constant
on all the facets of $P$. Hence, $\mathcal{C}(P,(\Lambda_\tau)_\mu)\simeq \mathcal{C}(3,1)$, and $N(P,\Lambda_\tau)/\langle\mu\rangle\simeq S^3$.
\end{proof}
\begin{lemma}
For each $\tau\in\{\boldsymbol{v}_1+\boldsymbol{v}_2,\boldsymbol{v}_1+\boldsymbol{v}_3,
\boldsymbol{v}_2+\boldsymbol{v}_3\}$ the complex $\mathcal{C}(P,\Lambda_\tau)$ is a bipartite Hamiltonian
subcomplex in $\partial P$ defined by a~set of disjoint simple edge-cycles containing all~the~vertices~of~$P$.
\end{lemma}
\begin{proof}
The subcomplex is bipartite and Hamiltonian by Proposition \ref{propo:taubpH}. Since  $N(P,\Lambda_\tau)$ is~a~closed manifold
and $\Lambda_\tau$ has rank $2$, the complex $\mathcal{C}(P,\Lambda_\tau)$  has no vertices (at a vertex there should 
be three linearly independent vectors). Thus, $\mathcal{C}^1(P,\Lambda_\tau)$ consists of a disjoint set of 
circles, and each vertex of $P$ lies on exactly one circle since $\mathcal{C}(P,\Lambda_\tau)\subset \partial P$ is a 
Hamiltonian subcomplex.
\end{proof}
\begin{proposition}
A disjoint set of simple edge-cycles containing all the vertices of $P$ defines
a bipartite Hamiltonian subcomplex $C$ in $\partial P$ if and only if for any cycle $\gamma$ and each connected 
component of $\partial P\setminus\gamma$ the number of vertices on $\gamma$ such that the
edge of $P$ incident to this vertex and not lying on $\gamma$ lies in the corresponding component is even. 
\end{proposition}
\begin{proof}
Indeed, if $C$ is bipartite Hamiltonian, then the facets of $P$ lying in each facet $G$ of $C$ can be nicely colored in two
colors (nicely means that adjacent facets have different colors).
Then the~colors along each component $\gamma$ of $\partial G$ alter, therefore there is an~even number 
of vertices on $\gamma$ with the~edge going inside $G$. On the~other hand, if for each component $\gamma$ of $\partial G$
there is an even number of such vertices, we can join these vertices outside $G$ by a set of simple 
piecewise linear paths. We obtain a disjoint set of simple piecewise linear closed curves on $\partial P$. Each curve 
divides $\partial P$ into two connected components homeomorphic to disks, hence the adjacency graph of the complement
to this set of curves is a tree. Thus, the complement can be nicely colored in two colors, and this coloring 
restricts to a nice coloring of $G$.
\end{proof}
\begin{corollary}\label{cor:4cp}
For a simple $3$-polytope $P$ the following conditions are equivalent:
\begin{enumerate}
\item $P$ has a coloring of facets in at most $4$ colors such that 
adjacent facets have different colors;
\item there is a disjoint set of simple edge-cycles such that
\begin{itemize}
\item any vertex of $P$ lies in exactly one cycle;
\item for any cycle $\gamma$ from the set and any of the two connected components of $\partial P\setminus\gamma$
the number of vertices on $\gamma$ with the edge going inside this component is even.
\end{itemize}
\end{enumerate} 
\end{corollary}
\begin{remark}
In \cite{B1913} G.~D.~Birkhoff reduced the four color problem 
to the family of simple $3$-polytopes $P$ such that $P$ is different from the simplex,
has no $3$- and $4$-belts and each $5$-belt surrounds a facet, where a~$k$-{\it belt} is a cyclic sequence of~$k$~facets 
such that two facets of this sequence are adjacent if and only if they are successive and no three facets
have a common vertex. We call such polytopes {\it strongly Pogorelov} (see~\cite{E19}). Then Corollary \ref{cor:4cp}
implies that the four color theorem is equivalent to the fact that any strongly Pogorelov polytope 
satisfies condition (2) of this corollary.
\end{remark}

\subsection{Subspaces of $N(P,\Lambda)$ corresponding to faces of $\mathcal{C}(P,\Lambda)$}\label{ssec:faces}
As it was mentioned in~Lemma \ref{lem:or} any $k$-face $G$ of $\mathcal{C}(P,c)$ is
an~orientable topological $k$-manifold, perhaps with a boundary. Moreover, Lemma \ref{lem:Mc} implies that 
$\partial G$ consists of faces of~$\mathcal{C}(P,c)$ lying in $G$. If $G$ is a connected component 
of the intersection $G_{i_1}\cap\dots\cap G_{i_{n-k}}$, then $\partial G$ consists 
of connected components of nonempty intersections $G\cap G_{j_1}\cap\dots\cap G_{j_l}$, 
$\{i_1<\dots<i_{n-k}\}\cap \{j_1<\dots< j_l\}=\varnothing$.
\begin{lemma}\label{lem:cc}
Connected components of $G\cap G_{j_1}\cap\dots\cap G_{j_l}$ are exactly connected components
of~$G_{i_1}\cap\dots\cap G_{i_{n-k}}\cap G_{j_1}\cap\dots\cap G_{j_l}$ lying in $G$.
\end{lemma}
\begin{proof}
Let $G_{i_1}\cap\dots\cap G_{i_{n-k}}=\bigsqcup\limits_{s} G_{\alpha_s}$,
$G_{j_1}\cap\dots\cap G_{j_l}=\bigsqcup\limits_{t} G_{\beta_t}$, and 
$G_{\alpha_s}\cap G_{\beta_t}=\bigsqcup\limits_{w}G_{s,t,w}$ 
be~the~decompositions into connected components. Then
$$
G_{i_1}\cap\dots\cap G_{i_{n-k}}\cap G_{j_1}\cap\dots\cap G_{j_l}=
\bigsqcup\limits_{s,t,w}G_{s,t,w}
$$
is a decomposition into a~disjoint union of closed connected sets. 
Each set in each union is~a~union of~faces of~$P$, hence
the number of sets in each union is finite. Therefore, each set   $G_{s,t,w}$ is~open 
in~the~topology induced to $G_{i_1}\cap\dots\cap G_{i_{n-k}}\cap G_{j_1}\cap\dots\cap G_{j_l}$
and it is a connected component of this set.
If $G=G_{\alpha_1}$, then by the same argument 
$G\cap G_{j_1}\cap\dots\cap G_{j_l}=\bigsqcup\limits_{t,w}G_{1,t,w}$ is~the~decomposition
into connected components. This finishes the proof. 
\end{proof}

\begin{lemma}\label{lem:intG}
For any faces $G'$ and $G''$ of $\mathcal{C}(P,c)$ if their intersection is nonempty, then it~is~a~disjoint union 
of~faces of $\mathcal{C}(P,c)$ of equal dimensions.
\end{lemma}
\begin{proof}
Indeed, if $G'$ is a connected component of~$G_{i_1}\cap\dots,\cap G_{i_a}\cap G_{i_{a+1}}\cap \dots\cap G_{i_b}$, 
and $G''$ is~a~connected component of~$G_{i_1}\cap\dots\cap G_{i_a}\cap G_{i_{b+1}}\cap \dots\cap G_{i_c}$,
where $G_{i_1}$, $\dots$, $G_{i_a}$, $\dots$, $G_{i_b}$, $\dots$, $G_{i_c}$ are different facets of~$\mathcal{C}(P,c)$,
then $G'\cap G''\subset G_{i_1}\cap \dots\cap G_{i_c}$, where the latter set is a disjoint union of its connected components --
faces of $\mathcal{C}(P,c)$ of the same dimension. Again, as in the proof of~Lemma~\ref{lem:cc} let
$G_{i_1}\cap\dots,\cap G_{i_a}\cap G_{i_{a+1}}\cap \dots\cap G_{i_b}=\bigsqcup\limits_{s} G_{\alpha_s}$,
$G_{i_1}\cap\dots\cap G_{i_a}\cap G_{i_{b+1}}\cap \dots\cap G_{i_c}=\bigsqcup\limits_{t} G_{\beta_t}$, and 
$G_{\alpha_s}\cap G_{\beta_t}=\bigsqcup\limits_{w}G_{s,t,w}$ be 
the~decompositions into connected components. Then
\begin{gather*}
G_{i_1}\cap \dots\cap G_{i_c}=\\
\left(G_{i_1}\cap\dots,\cap G_{i_a}\cap G_{i_{a+1}}\cap \dots\cap G_{i_b}\right)\cap
\left(G_{i_1}\cap\dots\cap G_{i_a}\cap G_{i_{b+1}}\cap \dots\cap G_{i_c}\right)=\\
\bigsqcup\limits_{s,t,w}G_{s,t,w}
\end{gather*}
is a decomposition into a~disjoint union of closed connected sets. 
Each set in each union is~a~union of~faces of~$P$, hence
the number of sets in each union is finite. Therefore, each set   $G_{s,t,w}$ is open in the 
topology induced to $G_{i_1}\cap \dots\cap G_{i_c}$
and it is a connected component of this set.
If~$G'=G_{\alpha_1}$ and $G''=G_{\beta_1}$ then by the same argument 
$G'\cap G''=\bigsqcup\limits_{w}G_{1,1,w}$ is~the~decomposition
into connected components. This finishes the proof. 

\end{proof}
\begin{definition}
For a $k$-face $G$ of $\mathcal{C}(P,c)$ we will call by facets of $G$ {\it connected components} of~nonempty 
intersections of $G\cap G_j$, $G\not\subset G_j$, and by faces of $G$ connected components of~intersection of its facets. 
\end{definition}
\begin{corollary}\label{cor:fG}
Faces of $G$ are exactly faces of $\mathcal{C}(P,c)$ lying in $G$.
\end{corollary}

The vector-coloring $\Lambda\colon\{F_1,\dots, F_m\}\to\mathbb Z_2^r$ 
of $P$ induces the vector-coloring $\Lambda_G\colon\{\widetilde{G}_1,\dots,\widetilde{G}_q\}\to \mathbb Z_2^r\to
\mathbb Z_2^r/\langle \Lambda_{i_1},\dots,\Lambda_{i_{n-k}}\rangle$ of facets of any $k$-face 
$G\subset G_{i_1}\cap\dots\cap G_{i_{n-k}}\subset \mathcal{C}(P,\Lambda)$: a connected component of
$G\cap G_j$ is mapped to $[\Lambda_j]$. 
\begin{definition}
Define $V_G=\langle \Lambda_G(\widetilde{G}_j),j=1,\dots,q\rangle\subset \mathbb Z_2^r/\langle \Lambda_{i_1},\dots,\Lambda_{i_{n-k}}\rangle $ and
$$
N(G,\Lambda_G)=G\times V_G/\sim,\text{ where }(p,t)\sim (q,s)\Leftrightarrow p=q\text{ and }t-s\in 
\langle\Lambda_G(\widetilde{G}_j)\colon p\in \widetilde{G}_j\rangle.
$$
\end{definition}

\begin{proposition}\label{prop:Gk}
Let $\Lambda$ be a vector-coloring of rank $r$ of a simple $n$-polytope $P$ such that 
$N(P,\Lambda)$ is a closed topological manifold, and let 
$\pi_{\Lambda}\colon N(P,\Lambda)\to P$ be the projection. 
Then for~any $k$-face $G$ of $\mathcal{C}(P,\Lambda)$ the preimage $\pi_{\Lambda}^{-1}(G)$
is a closed topological $k$-dimensional submanifold (locally defined in some coordinate system
as the intersection of $(n-k)$ coordinate hyperplanes) in $N(P,\Lambda)$ homeomorphic to 
a disjoint union of $|\left(\mathbb Z_2^r/\langle \Lambda_{i_1},\dots,\Lambda_{i_{n-k}}\rangle\right)/V_G|=2^{r-(n-k)-\dim V_G}$
copies of the~manifold $N(G,\Lambda_G)$.
\end{proposition}
\begin{proof}
Indeed, $N(P,\Lambda)=P\times \mathbb Z_2^r/\sim$,  where $(p,a)\sim(q,b)$  if and only if $p=q$  and 
$a-b\in\langle \Lambda_i\colon p\in F_i\rangle=\langle \Lambda_j\colon p\in G_j\rangle$.
Also $\pi_{\Lambda}^{-1}(G)=G\times \mathbb Z_2^r/\sim$, where $(p,a)\sim(q,b)$  if and only if $p=q$  and 
$a-b\in\langle \Lambda_j\colon p\in G_j\rangle$. But each point $p\in G$ lies in $G_{i_1}\cap\dots\cap G_{i_{n-k}}$. Hence,
$$
\mathbb Z_2^r/\langle\Lambda_j\colon p\in G_j\rangle\simeq 
(\mathbb Z_2^r/\langle \Lambda_{i_1},\dots,\Lambda_{i_{n-k}}\rangle)/
(\langle \Lambda_j\colon p\in G_j\rangle/\langle \Lambda_{i_1},\dots,\Lambda_{i_{n-k}}\rangle).
$$
This space is~a~disjoint union~of cosets modulo  $V_G$. Starting from the point  $[p,a]\in \pi_{\Lambda}^{-1}(G)$
and moving inside this space we can reach by a path only the points $[q,b]$, where $a-b\in V_G$. Hence, 
$\pi_{\Lambda}^{-1}(G)$ is a disjoint union of 
$|\left(\mathbb Z_2^r/\langle \Lambda_{i_1},\dots,\Lambda_{i_{n-k}}\rangle\right)/V_G|$ connected
components, and~each component is homeomorphic to $N(G,\Lambda_G)$. Since $N(P,\Lambda)$
is a closed manifold, for~each point $p\in \partial P$, which 
belongs to exactly $l$ facets $G_{j_1}$, $\dots$, $G_{j_l}$, the vectors  $\Lambda_{j_1},\dots, \Lambda_{j_l}$ are~linearly
independent. By Lemma \ref{lem:Mc} $p$ has a neighbourhood in $P$ homeomorphic 
to~$\mathbb R^l_{\geqslant}\times\mathbb R^{n-l}$.
Then in $N(P,\Lambda)$ for the point $p\times a$ these neighbourhoods are glued to~the~neighbourhood $U$ 
homeomorphic to~$\mathbb R^l\times \mathbb R^{n-l}$. Indeed, in $p\times a$ the copies $P\times (a+\varepsilon_1\Lambda_{j_1}+
\dots+\varepsilon_l\Lambda_{j_l})$, $\varepsilon_s=\pm1$, are glued locally as the sets 
$\{\varepsilon_1y_1\geqslant 0,\dots, \varepsilon_ly_l\geqslant 0\}$,  where the addition of the vector $\Lambda_{j_s}$ corresponds to~the~operation $y_s\to -y_s$.  At each point of  $\pi_{\Lambda}^{-1}(G)$ we may take $j_1=i_1$, $\dots$, $j_{n-k}=i_{n-k}$
to~see that $\pi_{\Lambda}^{-1}(G)\cap U$ is defined by equations $y_1=\dots=y_{n-k}=0$. This finishes the proof.
\end{proof}
\begin{proposition}\label{prop:kdefk+1}
Let $C=\mathcal{C}(P,c_2)\subset\mathcal{C}(P,c_1)$ be~a~proper 
Hamiltonian subcomplex and $M_1$, $\dots$, $M_s$
be the set of its defining $(n-2)$-faces. Then facets of~each $M_q$ are connected components of~intersections 
of~$M_q$ with facets~of~$C$ not containing $M_q$. Moreover, each $k$-face $G$ of $M_q$, $0\leqslant k\leqslant n-2$, 
is~a~connected component of~intersection of~$M_q$ with a~unique $(k+1)$-face $\widetilde{G}$ of~$C$ containing $G$.
\end{proposition}
\begin{proof}
By definition facets of $M_q\subset G_i\cap G_j$ are connected components of~intersections of~$M_q$
with facets $G_k$ of~$\mathcal{C}(P,c_1)$ different from $G_i$ and $G_j$. 
By Lemma~\ref{lem:Hamsc} for $G_i\cap G_j\cap G_k\ne\varnothing$
the~facets $G_i$ and $G_j$ lie in the same facet $\widetilde{G}_a$ of $\mathcal{C}(P,c_2)$, and 
$G_k$ lies in another facet $\widetilde{G}_b$. Then 
$$
M_q\cap G_k\subset M_q\cap \widetilde{G}_b=M_q\cap\left(\bigcup\limits_{G_l\subset \widetilde{G}_b}G_l\right)=
\bigcup\limits_{G_l\subset \widetilde{G}_b\colon M_q\cap G_l\ne\varnothing}M_q\cap G_l.
$$
The latter union is disjoint since $(M_q\cap G_{l_1})\cap (M_q\cap G_{l_2})\subset G_i\cap G_j\cap G_{l_1}\cap G_{l_2}=\varnothing$
by~Lemma~\ref{lem:Hamsc}(3). Each $M_q\cap G_l$ has a~finite set of~connected components since 
it~is~a~union of~faces~of~$P$. We have a disjoint union of a finite collection of connected components of all $M\cap G_l$, 
$G_l\subset \widetilde{G}_b$. Then each connected component is also a connected component of~the~union. 
This finishes the~proof of~the~first statement. 

Each $k$-face $G$ of~$M_q\subset G_i\cap G_j$ is~a~connected component 
of~$M_q\cap G_{i_3}\cap\dots \cap G_{i_{n-k}}$ by~Lemma~\ref{lem:cc}. 
By Lemma~\ref{lem:Hamsc} $G_i, G_j\subset\widetilde{G}_{j_2}$, $G_{i_3}\subset \widetilde{G}_{j_2}$,
$\dots$, $G_{i_{n-k}}\subset \widetilde{G}_{j_{n-k}}$ for different facets $\widetilde{G}_{j_2}$, $\dots$, $\widetilde{G}_{j_{n-k}}$
of~$\mathcal{C}(P,c_2)$. 
Then $M_q\cap \widetilde{G}_{j_2}\cap \dots\cap \widetilde{G}_{j_{n-k}}$ is equal to
the intersection of $M_q$ with a connected component $\widetilde{G}$ 
of~$\widetilde{G}_{j_2}\cap \dots\cap \widetilde{G}_{j_{n-k}}$
and it is equal to
\begin{gather*}
M_q\cap \widetilde{G}_{j_3}\cap \dots\cap \widetilde{G}_{j_{n-k}}=\\
M_q\cap\left(\bigcup\limits_{G_{l_3}\subset \widetilde{G}_{j_3}}G_{l_3}\right)\cap\dots\cap
\left(\bigcup\limits_{G_{l_{n-k}}\subset \widetilde{G}_{j_{n-k}}}G_{l_{n-k}}\right) =\\
\bigcup\limits_{G_{l_3}\subset \widetilde{G}_{j_3},\dots, G_{l_{n-k}}\subset \widetilde{G}_{j_{n-k}}}M_q\cap G_{l_3}\cap\dots\cap
G_{l_{n-k}}, 
\end{gather*}
where any two nonempty sets in the union are disjoint, since at some position $j$ the indices $l_j'$ and $l_j''$ are different
and $G_i\cap G_j\cap G_{l_j'}\cap G_{l_j''}=\varnothing$. Then each connected component of each set is a connected component of the union. In particular, for $l_3=i_3$, $\dots$, $l_{n-k}=i_{n-k}$. Thus, $G$ is a connected component of
intersection of~the~$(k+1)$-face $\widetilde{G}$ of $C$ with $M_q$, and $G\subset \widetilde{G}$.
On~the~other hand, let~$G$ be~a~connected component of intersection of some $(k+1)$-face $\widetilde{G}'$ of $C$ with $M_q$,
where $G\subset \widetilde{G}'$. Then $\widetilde{G}'$ is~a~connected component of some intersection
$\widetilde{G}_{u_2}\cap \dots\cap\widetilde{G}_{u_{n-k}}$. Then by the argument in the beginning 
of the proof of Lemma~\ref{lem:Hamsc} without loss of generality we~may~assume that
$G_i$, $G_j\subset \widetilde{G}_{u_2}$, $G_{i_3}\subset  \widetilde{G}_{u_3}$, $\dots$, 
$G_{i_{n-k}}\subset  \widetilde{G}_{u_{n-k}}$. Then $u_2=j_2$, $\dots$, $u_{n-k}=j_{n-k}$, and $\widetilde{G}=\widetilde{G}'$.
\end{proof}
\begin{proposition}\label{prop:Gtilde}
Let $C=\mathcal{C}(P,c_2)\subset\mathcal{C}(P,c_1)$ be~a~proper Hamiltonian subcomplex.
Then any $k$-face $\widetilde{G}\subset \mathcal{C}(P,c_2)$ is a~union 
$\widetilde{G}=\bigcup_{\alpha}G_\alpha$ of~$k$-faces of~$\mathcal{C}(P,c_1)$, 
where the intersection of any three faces in the union is empty,
while each nonempty intersection of two faces in~the~union is~a~disjoint union of~$(q-1)$-faces of defining faces. 
In particular, the adjacency graph $\Gamma(\widetilde{G})$ is~well-defined. Its vertices are faces $G_\alpha$
and edges correspond to connected components of intersections of pairs of such faces. The $(k-1)$-faces
of~defining faces corresponding to~the~edges of~$\Gamma(\widetilde{G})$ are pairwise disjoint.
\end{proposition}
\begin{proof}
By definition  $\widetilde{G}$ is a connected component of~the intersection 
$\widetilde{G}_{j_1}\cap\dots\cap \widetilde{G}_{j_{n-k}}$ of  pairwise different facets.
We have 
\begin{gather*}
\widetilde{G}_{j_1}\cap \dots\cap \widetilde{G}_{j_{n-k}}=\left(\bigcup\limits_{G_{l_1}\subset \widetilde{G}_{j_1}}G_{l_1}\right)\cap
\dots\cap \left(\bigcup\limits_{G_{l_{n-k}}\subset \widetilde{G}_{j_{n-k}}}G_{l_{n-k}}\right)=\\
\bigcup\limits_{G_{l_1}\subset \widetilde{G}_{j_1},\dots, G_{l_{n-k}}\subset \widetilde{G}_{j_{n-k}}}G_{l_1}\cap \dots\cap G_{l_{n-k}}.
\end{gather*}
The union on the right is a union of $k$-faces of $\mathcal{C}(P,c_1)$, which are connected components 
of~nonempty intersections. Moreover, if two such faces intersect, then they are connected components
of~two different intersections $G_{l_1}\cap \dots\cap G_{l_{n-k}}$ and $G_{l_1'}\cap \dots\cap G_{l_{n-k}'}$.
If $l_a\ne l_a'$ and $l_b\ne l_b'$, then 
$$
\left(G_{l_1}\cap \dots\cap G_{l_{n-k}}\right)\cap \left(G_{l_1'}\cap \dots\cap G_{l_{n-k}'}\right)\subset 
\left(G_{l_a}\cap G_{l_a'}\right)\cap \left(G_{l_b}\cap G_{l_b'}\right)=\varnothing.
$$ 
Thus, if faces intersect, then $l_a=l_a'$ for all indices $a$ but one, say $b$. Then
the intersection is a disjoint union of $(k-1)$-faces of defining faces lying in $G_{l_b}\cap G_{l_b'}$. 
If three faces intersect, then either for all $a$ but one $l_a=l_a'=l_a''$, and 
$l_b\ne l_b'\ne l_b''\ne l_b$, or there are two indices $b_1\ne b_2$ such that $l_{b_1}\ne l_{b_1}'$
and $l_{b_2}\ne l_{b_2}''$. In the first case the intersection lies 
in~$(G_{l_{b_1}}\cap G_{l_{b_1}'})\cap (G_{l_{b_1}}\cap G_{l_{b_1}''})=\varnothing$,
and in~the~second -- in~$(G_{l_{b_1}}\cap G_{l_{b_1}'})\cap (G_{l_{b_2}}\cap G_{l_{b_2}''})=\varnothing$. 
\end{proof}

\begin{corollary}\label{cor:GneMq}
Let $C=\mathcal{C}(P,c_2)\subset\mathcal{C}(P,c_1)$ be~a~proper Hamiltonian subcomplex.
If a $k$-face $G\subset \mathcal{C}(P,c_1)$ does not intersect defining faces, then
$G$ is also a~$k$-face of $\mathcal{C}(P,c_2)$ of the same combinatorial type. 
\end{corollary}
\begin{proof}
Indeed, $G$ lies in some $k$-face $\widetilde{G}$ of~$\mathcal{C}(P,c_2)$. 
Since $G$ does not intersect defining faces, we~have $G=\widetilde{G}$.
By the same argument each face of $G$ is a face of $\widetilde{G}$. On the other hand
each face of $\widetilde{G}$ is a subset of $G$ not intersecting defining faces. Hence, it
is a face of $G$. 
\end{proof}

\subsection{Bipartite Hamiltonian subcomplexes and two-sheeted branched coverings}\label{ssec:bc}
Remind that by $\pi_{\Lambda}$ we denote the projection $N(P,\Lambda)\to P$.
\begin{proposition}\label{prop:Ham2sh}
Let $\Lambda$ be a vector-coloring of rank $r$ of a simple $n$-polytope $P$ such
that $N(P,\Lambda)$ is a closed topological manifold and let 
$\tau\in\mathbb Z_2^r\setminus\{0\}$ be an involution such that 
$N(P,\Lambda)/\langle\tau\rangle=N(P,\Lambda_\tau)$ is also a closed topological manifold.
Let $M_1^{n-2}$, $\dots$, $M_s^{n-2}$ be defining faces of the bipartite Hamiltonian subcomplex 
$\mathcal{C}(P,\Lambda_\tau)\subset\mathcal{C}(P,\Lambda)$ and $M=M_1^{n-2}\sqcup\dots\sqcup M_s^{n-2}$. 
Then for the projection \linebreak $\pi\colon N(P,\Lambda)\to N(P,\Lambda_\tau)$ 
\begin{itemize}
\item the restriction $N(P,\Lambda)\setminus \pi_{\Lambda}^{-1}(M)\to 
N(P,\Lambda_\tau)\setminus\pi_{\Lambda_\tau}^{-1}(M)$ is~a~$2$-sheeted covering;
\item the restriction $\pi_{\Lambda}^{-1}(M)\to\pi_{\Lambda_\tau}^{-1}(M)$ it one-to-one
and for each point $x\in \pi_{\Lambda}^{-1}(M)$ and its image $y$
there are neighbourhoods such that
\begin{itemize}
\item each neighbourhood is homeomorphic to $\mathbb C\times\mathbb R^{n-2}$;
\item the preimage of $M$ corresponds to the set $\{0\}\times \mathbb R^{n-2}$;
\item the mapping $N(P,\Lambda)\to N(P,\Lambda_\tau)$ has the form 
$(\boldsymbol{z},\boldsymbol{x})\to (\boldsymbol{z}^2,\boldsymbol{x})$.
\end{itemize}
\end{itemize}
In particular, $\pi$ is~a~$2$-sheeted branched  covering with the branch set 
$\pi_{\Lambda_\tau}^{-1}(M)$ (see the definition in \cite{G19}). 
Moreover, $\pi_{\Lambda_\tau}^{-1}(M)\simeq \pi_{\Lambda}^{-1}(M)=
\bigsqcup\limits_{q=1}^s\pi_{\Lambda}^{-1}(M_q)$, where 
$\pi_{\Lambda}^{-1}(M_q)$ is a disjoint union of~$2^{r-2-\dim V_{M_q}}$ copies of the closed $(n-2)$-manifold 
$N(M_q,\Lambda_{M_q})$.
\end{proposition}
\begin{proof}
Indeed, for each point $p\notin M$ all the different facets $G_i$  of $\mathcal{C}(P,\Lambda)$ containing $p$
lie in~different facets $\widetilde{G}_j$ of $\mathcal{C}(P,\Lambda_\tau)$. As it was mentioned 
in the proof of Proposition~\ref{prop:Gk} each point $[p\times a]\in N(P,\Lambda)$ 
such that $p$ belongs to exactly $l$ facets $G_{j_1}$, $\dots$, $G_{j_l}$
has a neighbourhood $U([p\times a])$ homeomorphic to $\mathbb R^n$. Moreover, there are
$2^{r-l}$ such neighbourhoods corresponding to cosets in 
$\mathbb Z_2^r/\langle \Lambda_{j_1},\dots,\Lambda_{j_l}\rangle$ and they can be chosen to be disjoint.
For $[p\times [a]]\in N(P,\Lambda_\tau)$ we~have similar $2^{r-1-l}$ neighbourhoods corresponding
to cosets in $(\mathbb Z_2^r/\langle\tau\rangle)/\langle [\Lambda_{j_1}],\dots,[\Lambda_{j_l}]\rangle$.
Each neighbourhood of $N(P,\Lambda)$ corresponding to a coset $b+\langle \Lambda_{j_1},\dots,\Lambda_{j_l}\rangle$ 
is mapped to~the~neighbourhood corresponding to the coset 
$[b]+\langle [\Lambda_{j_1}],\dots,[\Lambda_{j_l}]\rangle$ and this mapping is a homeomorphism.
Moreover,  the coset $[b]+\langle [\Lambda_{j_1}],\dots,[\Lambda_{j_l}]\rangle$
corresponds to exactly two cosets $b+\langle \Lambda_{j_1},\dots,\Lambda_{j_l}\rangle$ and 
$(b+\tau)+\langle \Lambda_{j_1},\dots,\Lambda_{j_l}\rangle$.  

Now consider a~point  $p\in M_q^{n-2}\subset G_{j_1}\cap G_{j_2}$ such that $G_{j_1}$ and $G_{j_2}$
lie in the same facet $\widetilde{G}_{i_2}$  of $\mathcal{C}(P,\Lambda_\tau)$ and $p$ lies in exactly 
$l$ facets $G_{j_1}$, $G_{j_2}$, $\dots$, $G_{j_l}$. By Lemma~\ref{lem:Hamsc} the sets $G_{j_1}\cup G_{j_2}$,
$G_{j_3}$, $\dots$, $G_{j_l}$ lie in $l-1$ different facets $\widetilde{G}_{i_2}$, $\widetilde{G}_{i_3}$, 
$\dots$, $\widetilde{G}_{i_l}$. By Lemma~\ref{lem:Mc} the point $p$ in $P$ has a~neighbourhood 
$$
U\simeq\mathbb R^l_{\geqslant }\times\mathbb R^{n-l}=\{(y_1,\dots,y_n)\in\mathbb R^n\colon y_1\geqslant 0,\dots,y_l\geqslant 0\},
$$ 
corresponding to $\mathcal{C}(P,\Lambda)$, where the facet $G_{j_p}$
corresponds to~the~hyperplane $y_p=0$. Set $\boldsymbol{z}=y_1+i y_2$ and
$\boldsymbol{x}=(y_3,\dots, y_n)$. The mapping $(\boldsymbol{z}, \boldsymbol{x})\to (\boldsymbol{z}^2, \boldsymbol{x})$
defines a homeomorphism 
\begin{gather*}
\mathbb R^l_{\geqslant }\times\mathbb R^{n-l}\to\mathbb R\times \mathbb R_{\geqslant}\times 
\mathbb R^{l-2}_{\geqslant}\times\mathbb R^{n-l}=\\
\{(\widetilde{y}_1,\dots,\widetilde{y}_n)
\in\mathbb R^n\colon \widetilde{y}_2\geqslant 0,\dots,\widetilde{y}_l\geqslant 0\}
\simeq \mathbb R^{l-1}_{\geqslant}\times\mathbb R^{n-l+1},
\end{gather*}
and under the composition of homeomorphisms the facets $\widetilde{G}_{i_p}$ are mapped \linebreak to~the~hyperplanes
$\widetilde{y}_p=0$, $p=2$, $\dots$, $l$. In both coordinate systems $M_q^{n-2}$ is defined by the condition  
$y_1=y_2=0$ and  $\widetilde{y}_1=\widetilde{y}_2=0$.
In $N(P,\Lambda)$ the neighbourhoods $U$ are glued to~$2^{r-l}$ neighbourhoods $U([p,a])$ corresponding 
to~cosets in~$\mathbb Z_2^r/\langle \Lambda_{j_1},\dots,\Lambda_{j_l}\rangle$. 
In $N(P,\Lambda_\tau)$ the neighbourhoods $U$  are glued to~$2^{r-l}=2^{(r-1)-(l-1)}$ neighbourhoods $U([p,[a]])$ 
corresponding to~cosets in~$(\mathbb Z_2^r/\langle\tau\rangle)/\langle [\Lambda_{j_2}],\dots,[\Lambda_{j_l}]\rangle$.
We have 
$$
\mathbb Z_2^r/\langle \Lambda_{j_1},\dots,\Lambda_{j_l}\rangle\simeq 
(\mathbb Z_2^r/\langle\tau\rangle)/\langle [\Lambda_{j_2}],\dots,[\Lambda_{j_l}]\rangle,
$$
since $\tau=\Lambda_{j_1}+\Lambda_{j_2}$.
Thus, a neighbourhood corresponding to~a coset $b+\langle \Lambda_{j_1},\dots,\Lambda_{j_l}\rangle$ 
is mapped to~the~neighbourhood corresponding to the coset 
$[b]+\langle [\Lambda_{j_2}],\dots,[\Lambda_{j_l}]\rangle$ and in the above coordinates the~sets
$\pi_{\Lambda}^{-1}(M)$ and $\pi_{\Lambda_\tau}^{-1}(M)$ are defined by the equations
$\boldsymbol{z}=0$ and $\widetilde{\boldsymbol{z}}=0$, and the~mapping has the~form 
$(\widetilde{\boldsymbol{z}}, \widetilde{\boldsymbol{x}})=(\boldsymbol{z}^2, \boldsymbol{x})$.  

The last statement follows from Proposition~\ref{prop:Gk}. 
\end{proof}

\subsection{$\mathcal{C}(n,k)$-subcomplexes and hyperelliptic involutions}\label{ssec:cnk}
\begin{definition}\label{def:CnkH}
We call by a~{\it $\mathcal{C}(n,k)$-subcomplex} $C$ of $\mathcal{C}(P,c)$ a~subcomplex 
$C=\mathcal{C}(P,c_2)\subset \mathcal{C}(P,c)$ such that $\mathcal{C}(P,c_2)\simeq \mathcal{C}(n,k)$.
A~{\it  $\mathcal{C}(n,k)$-subcomplex} is {\it Hamiltonian}, if $\mathcal{C}(P,c_2)$ is Hamiltonian. 
\end{definition}
\begin{corollary}\label{cor:HCnk}
Any $\mathcal{C}(n,k)$-subcomplex $C\subset\partial P$ corresponds to~a~subgroup $H(C)\subset \mathbb Z_2^m$
of~rank $m-k$ such that $\mathbb{R}\mathcal{Z}_P/H(C)\simeq S^n$. The subgroup is defined in $\mathbb Z_2^m$ by equations 
$\sum\limits_{F_j\subset G_i}x_j=0$, where $G_i$,
$i=1$, $\dots$, $k$, are facets of the subcomplex. 
\end{corollary}
\begin{proof} 
This follows from \cite[Construction 5.8]{E24}.
\end{proof}

Theorem \ref{th:n34} implies the following result.
\begin{corollary}\label{cor:cnk4}
For $n\leqslant 4$ the correspondence $C\to H(C)$  is a bijection between  $\mathcal{C}(n,k)$-subcomplexes $C\subset\partial P$
and subgroups $H\subset \mathbb Z_2^m$ of rank $m-k$ such that $\mathbb{R}\mathcal{Z}_P/H\simeq S^n$.
\end{corollary}

\begin{proposition}\label{prop:cnkbh}
Let $C=\mathcal{C}(P,c_2)\subset \mathcal{C}(P,c_1)$ be~a~proper Hamiltonian \linebreak $\mathcal{C}(n,k)$-subcomplex
and $(n,k)\ne (2,1)$. 
Then for any facet $\widetilde{G}$ of $C$ the adjacency graph $\Gamma(\widetilde{G})$ is a tree. In particular, the subcomplex $C$ is bipartite.
\end{proposition}
\begin{remark} For $(n,k)=(2,1)$ the adjacency graph of a single facet is a cycle. 
In this case the subcomplex $C$ is bipartite if and only if the cycle has even length.
\end{remark}
\begin{proof}[of Proposition~\ref{prop:cnkbh}]
For $n\leqslant 2$ the proof is straightforward. Assume that $n\geqslant 3$.
Any facet $\widetilde{G}$ of $C$ is either a topological disk for $k>1$, or a topological sphere for $k=1$. 
Consider a defining $(n-2)$-face $M\subset G_i\cap G_j$,
where $G_i,G_j\subset \widetilde{G}$ are facets of $\mathcal{C}(P,c_1)$. $M$ is~a~connected orientable $(n-2)$-manifold,
perhaps with a boundary. If $M$ has a boundary, then $\partial M$ is a union of $(n-3)$-faces. 
Since $C$ is Hamiltonian, any such a face lies in~the~$(n-2)$-skeleton of $\mathcal{C}(P,c_2)$. 
In particular, $\partial M\subset \partial \widetilde{G}$, which is a topological sphere for $k>1$, and 
$\partial \widetilde{G}=\varnothing$ for $k=1$. Consider a~point $p\in M$. Let $p$ belong
to~exactly $l$ facets $G_{j_1}$, $\dots$, $G_{j_l}$, where $G_{j_1}=G_i$ and $G_{j_2}=G_j$. 
By~Lemma~\ref{lem:Hamsc} the sets $G_{j_1}\cup G_{j_2}$,
$G_{j_3}$, $\dots$, $G_{j_l}$ lie in $l-1$ different facets $\widetilde{G}_{i_2}=\widetilde{G}$, $\widetilde{G}_{i_3}$, 
$\dots$, $\widetilde{G}_{i_l}$. As~in~the~proof of~Proposition~\ref{prop:Ham2sh} the point $p$ in $P$ has a~neighbourhood 
$$
U\simeq\{(\widetilde{y}_1,\dots,\widetilde{y}_n)
\in\mathbb R^n\colon \widetilde{y}_2\geqslant 0,\dots,\widetilde{y}_l\geqslant 0\}
\simeq \mathbb R^{l-1}_{\geqslant}\times\mathbb R^{n-l-1},
$$
where the facets $\widetilde{G}_{i_p}$ correspond to~the~hyperplanes
$\widetilde{y}_p=0$, $p=2$, $\dots$, $l$, and $M$ is defined by~the~equations~$\widetilde{y}_1=\widetilde{y}_2=0$.
Thus, $M$ is an~orientable polyhedral $(n-2)$-manifold locally flat embedded to~the~polyhedral $(n-1)$-disk $\widetilde{G}$ 
for $k>1$ or a polyhedral sphere $\widetilde{G}$ for $k=1$, and $\partial M=M\cap \partial \widetilde{G}$.

\begin{lemma}\label{lem:2cc}
The complement $\widetilde{G}\setminus M$ has exactly two connected components.
\end{lemma}
\begin{proof}
It is easy to see that the complement $\widetilde{G}\setminus M$ has at most two connected components, 
since it is valid locally at each point of $M$, and from each point of the complement we can reach such 
a~point by a simple piecewise linear path intersecting $M$ only at this point.

On the other hand, assume that there is only one connected component. If $\partial M\ne\varnothing$, then this is valid also 
for~the~double manifolds $DM=M\cup_{\partial M}M\subset D\widetilde{G}=
\widetilde{G}\cup_{\partial \widetilde{G}} \widetilde{G}\simeq S^{n-1}$.
But~$DM$ is an~orientable $(n-2)$-manifold, hence $H_{n-2}(DM)=\mathbb Z$. By the Alexander duality
$\widetilde{H}^0(\widetilde{G}\cup_{\partial \widetilde{G}} \widetilde{G}\setminus M\cup_{\partial M}M)\simeq \mathbb Z$. 
Hence, there are two connected components. A contradiction. 

If $\partial M=\varnothing$, then $M\cap \partial\widetilde{G}=\varnothing$ and for the space 
$\widetilde{G}'=\widetilde{G}/\partial \widetilde{G}\simeq S^n$ the complement $\widetilde{G}'\setminus M$
also has one connected component. But this contradicts the~Alexander duality. 
\end{proof}

\begin{corollary}\label{cor:1cc}
We have $M=G_i\cap G_j$.
\end{corollary}
\begin{proof}
Indeed, by Lemma~\ref{lem:2cc} $\widetilde{G}\setminus M$ has two connected components $C_1$ and $C_2$,
where ${\rm int}\, G_i\subset C_1$ and ${\rm int}\, G_j\subset C_2$. 
Moreover, $\overline{C_1}\cap \overline {C_2}=M$ since $M$ is locally flat embedded to $\widetilde{G}$. 
Then $G_i$ and $G_j$ have no common points lying outside $M$.   
\end{proof}
Lemma~\ref{lem:2cc} and  Corollary~\ref{cor:1cc} imply that the adjacency graph $\Gamma(\widetilde{G})$ is a tree.
Since it~is~valid for each $\widetilde{G}$, the subcomplex $C$ is bipartite.
\end{proof}

\begin{construction}[(A vector-coloring induced by a Hamiltonian\linebreak  $\mathcal{C}(n,k)$-subcomplex)]\label{con:LHC}
Proposition \ref{prop:cnkbh} and Example \ref{ex:vci} imply that any proper Hamiltonian  $\mathcal{C}(n,k)$-subcomplex
$C=\mathcal{C}(P,c_2)\subset \mathcal{C}(P,c_1)$ for $(n,k)\ne (2,1)$ induces a vector-coloring $\widetilde{\Lambda}_C$ of $P$ of rank $k+1$
defined up to a linear change of coordinates in $\mathbb Z_2^{k+1}$. We have 
$\mathcal{C}(P,\widetilde{\Lambda}_C)=\mathcal{C}(P,c_1)$ and $N(P,\widetilde{\Lambda}_C)$ is~a~hyperelliptic manifold 
with the~hyperelliptic involution $(0,0,\dots,1)$ (in terms of Construction \ref{con:indfrom}).  
A proper Hamiltonian  $\mathcal{C}(2,1)$-subcomplex 
induces such a vector coloring if and only if  $\mathcal{C}(P,c_1)$ has an even number of facets.
\end{construction}

Theorem \ref{th:n34} implies the following result.
\begin{theorem}\label{th:Hamsc}
Let $n\leqslant 4$ and  $\Lambda$ be a~vector-coloring of rank $r$ of~a~simple $n$-polytope $P$ such that 
$N(P,\Lambda)$ is~a~closed topological manifold. Then nonzero hyperelliptic involutions 
$\tau\in \mathbb Z_2^r\setminus\{0\}$ 
are in~bijection with proper Hamiltonian $\mathcal{C}(n,r-1)$-subcomplexes $C\subset\mathcal{C}(P,\Lambda)$
inducing $\Lambda$ (that is $\Lambda=\widetilde{\Lambda}_C$ up to a change of coordinates).
\end{theorem}
\begin{proof}
Indeed, if there is a proper Hamiltonian $\mathcal{C}(n,r-1)$-subcomplex $C\subset\mathcal{C}(P,\Lambda)$ 
such that $\Lambda=\widetilde{\Lambda}_C$ up to a change of coordinates, then in these coordinates 
the~involution $\tau_C$ given by the vector $e_r$ is hyperelliptic, since 
$$
C(P,\Lambda_C)\simeq C\simeq C(n,r-1)
$$ 
and 
$$
N(P,\Lambda)/\langle\tau\rangle\simeq N(P,\Lambda_\tau)\simeq N(P,\Lambda_C)\simeq S^n
$$ 
by Theorem~\ref{th:n34}.

On the other hand, if an involution $\tau\in \mathbb Z_2^r\setminus\{0\}$ is hyperelliptic, then 
$$
N(P,\Lambda_\tau)\simeq N(P,\Lambda)/\langle\tau\rangle\simeq S^n.
$$
By Theorem~\ref{th:n34} the~complex $C=C(P,\Lambda_\tau)$ is equivalent to~$C(n,r-1)$. Then 
different vectors  $\Lambda_\tau(F_j)$ are linearly independent and up to a change of coordinates 
$\Lambda_\tau=\Lambda_C$. It follows from \cite[Proposition 5.12]{E24} (see also Proposition~\ref{propo:taubpH}) that
$C\subset\mathcal{C}(P,\Lambda)$ is a~Hamiltonian subcomplex. Moreover,  $\Lambda(G_i)=\Lambda(G_j)+\tau$ 
for adjacent facets $G_i$ and $G_j$ of $\mathcal{C}(P,\Lambda)$ lying in the same facet of $C$. Thus, 
$\Lambda=\widetilde{\Lambda}_C$ up to a~change of coordinates, and  $\tau=\tau_C$. This finishes the~proof.
\end{proof}

\begin{definition}
For a~complex $C=\mathcal{C}(P,c)$ set $f_i(C)$ to~be~the~number of~its~ $i$-faces. By definition $f_n(C)=1$.
We have $f_{n-1}(C)=M$.
\end{definition}

\begin{proposition}\label{prop:treek}
Let $C=\mathcal{C}(P,c_2)\subset \mathcal{C}(P,c_1)$ be a~proper Hamiltonian \linebreak 
$\mathcal{C}(n,r)$-subcomplex. If~a~$k$-face 
$\widetilde{G}$ of $C$ is not a circle, then the~adjacency graph $\Gamma(\widetilde{G})$ of~$k$-faces 
$G$ of~$\mathcal{C}(P,c_1)$ lying in~$\widetilde{G}$ is~a~tree, where we call the faces adjacent if they have 
a~nonempty intersection.
\end{proposition}
\begin{proof}
Indeed, for $k=0$ this is trivial. For $k=1$ if $\widetilde{G}$ is not a circle, then it is an~edge subdivided 
by vertices of~$\mathcal{C}(P,c_1)$ into edges, and  $\Gamma(\widetilde{G})$ is~a~path.
Assume that $k>1$. Each face of~$C$ is~either a~topological disk or~a~sphere.
Let $\widetilde{G}$ be a~connected component of~$\widetilde{G}_{j_1}\cap\dots\cap \widetilde{G}_{j_{n-k}}$.
\begin{lemma}
If the intersection of two $k$-faces $G'$ and $G''$ of $\mathcal{C}(P,c_1)$ lying in $\widetilde{G}$
is nonempty, then $G'\cap G''$ is a $(k-1)$-face of $\mathcal{C}(P,c_1)$ lying in a defining $(n-2)$-face.
Moreover, $\widetilde{G}\setminus(G'\cap G'')$ consists of two connected components $C_1$ and $C_2$
such that $G'\cap G''=\overline{C_1}\cap \overline{C_2}$, and $G'\subset \overline{C_1}$, $G''\subset \overline{C_2}$.
\end{lemma}
Let $G'$ and $G''$ be connected components of~$G_{a_1}\cap\dots\cap G_{a_{n-k}}$ and $G_{b_1}\cap\dots\cap G_{b_{n-k}}$.
Then up~to~a~renumbering of~indices $G_{a_i},G_{b_i}\subset \widetilde{G}_{j_i}$ for all $i$.
By Lemma~\ref{lem:Hamsc} $a_i\ne b_i$ only for one $i$, say for $i=1$. Then 
$G'\cap G''\subset G_{a_1}\cap G_{b_1}\cap G_{b_2}\cap \dots\cap G_{b_{n-k}}$. In particular, each connected 
component $G$ of $G'\cap G''$ lies in a defining $(n-2)$-face.

The argument in the proof of~Proposition~\ref{prop:cnkbh} shows 
that $\widetilde{G}\setminus G$ has two connected components $C_1$ and $C_2$ with
$G=\overline{C_1}\cap \overline{C_2}$ and $G'\subset \overline{C_1}$, $G''\subset \overline{C_2}$.
Then $G'\cap G''=G$ and the adjacency graph $\Gamma(\widetilde{G})$ is~a~tree.
\end{proof}

\begin{proposition}\label{prop:fkform}
Let $C_2=\mathcal{C}(P,c_2)\subset \mathcal{C}(P,c_1)=C_1$ be a~proper Hamiltonian $\mathcal{C}(n,r)$-subcomplex
with the~set of~defining faces $M_1$, $\dots$, $M_s$. 
Then for $0\leqslant k\leqslant n-1$ 
\begin{gather}\label{eq:kk-1}
f_k(C_1)=f_k(C_2)-\delta_{r,n-1}\delta_{k,1}+\sum\limits_{q=1}^s\left(f_k(M_q)+f_{k-1}(M_q)\right)=\\
{r\choose n-k}+\delta_{r,n}\delta_{k,0}-\delta_{r,n-1}\delta_{k,1} +\sum\limits_{q=1}^s\left(f_k(M_q)+f_{k-1}(M_q)\right).
\end{gather}
where $\delta_{i,j}=1$ if $i=j$,  and $\delta_{i,j}=0$ if $i\ne j$.
\end{proposition}
\begin{proof}
Indeed, each $k$-face $G$ of $C_1$ either lies in some $M_q$ or it does not lie in defining faces but lies 
in~a~unique $k$-face $\widetilde{G}$ of~$C_2$. By Proposition~\ref{prop:treek} if $(r,k)\ne (n-1,1)$ 
the~$k$-faces $G$ of~the~latter type lying in $\widetilde{G}$ form a tree with edges corresponding to $(k-1)$-faces 
of defining faces lying in $\widetilde{G}$. 
By Proposition~\ref{prop:kdefk+1} each $(k-1)$-face of a defining face lies in a unique $k$-face $\widetilde{G}$ of $C_2$.
Thus, for $(r,k)\ne (n-1,1)$ the~number of $k$-faces $G\subset \widetilde{G}$ is one greater then the~number 
of $(k-1)$-faces inside $\widetilde{G}$
lying in defining faces. For $(r,k)=(n-1,1)$ the number of $1$-faces of~$C_1$ lying in~a~unique $1$-face
of~$C_2$ (which is a circle) is equal to the number of vertices of~defining faces. Also we have 
$f_k(C_2)={r\choose n-k}+\delta_{r,n}\delta_{k,0}$. This finishes the proof. 
\end{proof}

\begin{proposition}\label{prop:Mq}
Let $C_2=\mathcal{C}(P,c_2)\subset \mathcal{C}(P,c_1)=C_1$ be a~proper Hamiltonian $\mathcal{C}(n,r)$-subcomplex
with the~set of~defining faces $M_1$, $\dots$, $M_s$. Then for any defining face $M_q$ its facets can be colored 
in~$(r-1)$ colors in~such a~way that if~the~intersection of~a~set of facets is~nonempty, then their colors are pairwise different. 
\end{proposition}
\begin{proof}
Indeed, any defining face $M_q$ lies in some facet $\widetilde{G}_i$ of $C_2$ and its facets are connected 
components of intersections of $M_q$ with facets $\widetilde{G}_j\ne \widetilde{G}_i$. Then we can assign 
to~each connected component of~$M_q\cap \widetilde{G}_j$ the color $j$. If the intersection of facets of $M_q$
is nonempty, then by~construction they have different colors. 
\end{proof}

\subsection{Geometric hyperelliptic manifolds $N(P,\Lambda)$}\label{ssec:geom}
If $P$ is a right-angled $n$-polytope in some geometry $\mathbb X$ and $\Lambda$ is a linearly independent
vector-coloring of rank $r$, then $N(P,\Lambda)$ is a manifold with a geometric structure modelled on $\mathbb X$.

\begin{construction}\label{con:hypgeom}
Let $P$ be~a~right-angled $n$-polytope in some geometry $\mathbb X$. 
If $C\subset P$ is a~proper Hamiltonian $C(n,k)$-subcomplex, then there is a~canonical linearly independent
vector-coring $\widetilde{\Lambda}_C$ of rank $k+1$ induced by $C$. It is 
defined up to a linear change of coordinates in $\mathbb Z_2^{k+1}$. 
Then $N(P,\widetilde{\Lambda}_C)$ is a geometric hyperelliptic 
$n$-manifold with a~hyperelliptic involution in~$\mathbb Z_2^{k+1}$
corresponding to $C$.
\end{construction}
\begin{remark}
Construction \ref{con:hypgeom} is a~direct generalization 
to dimensions $n>3$ of the construction of geometric 
hyperelliptic $3$-manifolds in \cite{M90, VM99M, VM99S2} based on 
Hamiltonian cycles, theta- and $K_4$-subgraphs in $1$-skeletons of right-angled $3$-polytopes.
\end{remark}
\begin{remark}\label{rem:4hi}
Theorem~\ref{th:Hamsc} implies that for $n\leqslant 4$ and any linearly independent vector-coloring $\Lambda$
of~a right-angled $n$-polytope $P$ if $N(P,\Lambda)$ admits a~hyperelliptic involution in $\mathbb Z_2^r$,
then this involution, the coloring~$\Lambda$ and $N(P,\Lambda)$ can 
be~obtained by Construction \ref{con:hypgeom}. 
\end{remark}
In this section we will discuss which geometries $\mathbb X$ admit hyperelliptic manifolds obtained 
by~Construction~\ref{con:hypgeom}.

\begin{corollary}
Let $n>2$ and $C$ be a Hamiltonian $\mathcal{C}(n,r)$-subcomplex in $\partial P$. Then $n-1\leqslant r\leqslant n+1$ and 
the vertex set of $P$ is a disjoint union of vertices of $C$ (their number $V_{n,r}$ is: $V_{n,n-1}=0$, $V_{n,n}=2$, 
$V_{n,n+1}=n+1$) 
and vertices of~defining $(n-2)$-faces $M_1$, $\dots$, $M_s$. The~number of defining faces is~equal to $m-r$.
Moreover, for each $0\leqslant k\leqslant n-1$
$$
f_k(P)= {r\choose n-k}+\delta_{r,n}\delta_{k,0}-\delta_{r,n-1}\delta_{k,1} +\sum\limits_{q=1}^s\left(f_k(M_q)+f_{k-1}(M_q)\right).
$$
\end{corollary}
\begin{proof}
This follows directly from Proposition~\ref{prop:fkform}.

In particular, each vertex of $P$  either is~a~vertex of $C$ or lies on its $1$-face. Thus, $C$ should have $1$-faces. This
is possible only for $k\in\{n-1,n,n+1\}$. Moreover, $\mathcal{C}^1(n,n-1)$ is a circle, $\mathcal{C}^1(n,n-1)$ consists
of two vertices and $n$ multiple edges, and $\mathcal{C}^1(n,n+1)$ is the complete graph on $n$ vertices 
(see \cite[Example 8.17]{E24}). If a vertex $v$ of $P$ is a vertex of $C$, then it~is~the~intersection of $n$
different facets of $C$, and it does not lie in defining faces. If $v$ lies inside a $1$-face of $C$, then  
two facets of $P$ containing $v$ lie in the same facet of $C$. Then their intersection is a defining face.

For $k=n-1>1$ we have $m=f_{n-1}(P)={r\choose 1}+\sum\limits_{q=1}^sf_{n-2}(M_q)=r+s$. Thus, $s=m-r$
is~the~total number of defining faces.
\end{proof}
\begin{corollary}\label{cor:In}
The $n$-cube $I^n$ does not admit Hamiltonian $\mathcal{C}(n,k)$-subcomplexes for $n>3$. In particular,
there are no locally Euclidean hyperelliptic $n$-manifolds $N(P,\Lambda)$, $n>3$, obtained by~Construction~\ref{con:hypgeom}.
\end{corollary}
\begin{proof}
The cube has $2^n$~vertices and $2n$ facets. Any $(n-2)$-face is $I^{n-2}$ and has $2^{n-2}$ vertices.
If~$I^n$ has a~Hamiltonian $\mathcal{C}(n,k)$-subcomplex, then $2^n=V_{n,k}+(2n-k)2^{n-2}$,
where $n-1\leqslant k\leqslant n+1$. We have $2n-k\geqslant n-1$. Thus, the right part
is at least $V_{n,k}+(n-1)2^{n-2}$. If $n\geqslant 5$, then it~is~at~least $V_{n,k}+4\cdot 2^{n-2}=V_{n,k}+2^n\geqslant 2^n$,
where the equality holds only if $n=5$ and $k=n+1$. But in~this case $V_{n,k}=n+1=6$ and we obtain a contradiction.
If $n=4$, then we have $16=V_{n,k}+(8-k)\cdot 4$. For $k=3$ we have $16=0+5\cdot 4=20$, 
for $k=4$ we have $16=2+4\cdot 4=18$, and for $k=5$ we have $16=5+3\cdot 4=17$. A contradiction. 
\end{proof}
\begin{remark}
For $n\leqslant 3$ the cube $I^n$ admits Hamiltonian $\mathcal{C}(n,k)$-subcomplexes.
\end{remark}
\begin{corollary}\label{cor:n-2n-2}
Let and $C$ be a Hamiltonian $\mathcal{C}(n,n-1)$-subcomplex in $\partial P$.
Then any its~defining face is an $(n-2)$-polytope admitting a coloring of facets in $(n-2)$ colors
such that adjacent facets have different colors.  
\end{corollary}
\begin{remark}\label{rem:nn}
As is was proved in \cite{J01} a simple $n$-polytope $P$ admits a coloring of its facets in $n$
colors such that adjacent facets have different colors if and only if any $2$-face of $P$ has 
an~even number of~edges.
\end{remark}
\begin{corollary}\label{cor:nn-1}
Let and $C$ be a Hamiltonian $\mathcal{C}(n,n-1)$-subcomplex in $\partial P$. Then 
the vertices of $P$ lie on a disjoint set of even-gonal $2$-faces of defining faces.   
\end{corollary}
\begin{proof}
Indeed, by Corollary~\ref{cor:n-2n-2} and Remark~\ref{rem:nn} each defining face $M_q$ 
has a~coloring in~$(n-2)$ colors such that adjacent facets have different colors 
and each its $2$-gonal face has an even number of edges.
Moreover, for each choice of $(n-4)$ colors any nonempty intersection of facets of these colors
is an even-gonal $2$-face of $M_q$, and any vertex of $M_q$ lies on exactly one such a~$2$-face. 
\end{proof}
\begin{corollary}
If a small cover $N(P,\Lambda)$ over a $4$-polytope $P$ has a hyperelliptic involution in~$\mathbb Z_2^4$,
then this involution corresponds to a~Hamiltonian $\mathcal{C}(4,3)$-subcomplex inducing $\Lambda$
and the~vertices of~$P$ lie on a disjoint set of $(m-3)$ defining even-gons.
\end{corollary}
\begin{remark}
The first example of~a small cover $N(P,\Lambda)$ over a $4$-polytope $P$  with a~hyperelliptic involution in~$\mathbb Z_2^4$
was build by Alexei Koretskii, see \cite{K24B} and \cite{K24}.
\end{remark}
\begin{proposition}\label{prop:Ln}
If a $4$-polytope $P$ admits a Hamiltonian $\mathcal{C}(4,r)$-subcomplex, then $P$
has~at~least one triangular or quadrangular $2$-face. In particular, there are 
no~hyperbolic ($\mathbb X=\mathbb L^4$) hyperelliptic 
$4$-manifolds $N(P,\Lambda)$  obtained by~Construction~\ref{con:hypgeom}. Moreover,
for all~$n\geqslant 4$ there are no  hyperbolic ($\mathbb X=\mathbb L^n$) $n$-manifolds $N(P,\Lambda)$ obtained by~Construction~\ref{con:hypgeom}.
\end{proposition}
\begin{proof}
Let $p_k$ be the number of $k$-gonal $2$-faces of $P$, and $f_i$ be the total number of its $i$-faces 
(in particular, $f_3=m$). We have the Euler-Poincare formula $f_0-f_1+f_2-f_3=0$. Also $4f_0=2f_1$ since $P$ is simple.
Also $\sum\limits_k kp_k={4\choose 2} f_0=6f_0$.
Then $f_2=f_1-f_0+f_3=f_0+f_3=f_0+m$, and 
$$
\sum\limits_kp_k=f_2=f_0+m=\frac{1}{6}\sum\limits_k kp_k+m.
$$
Thus, $\sum\limits_k6p_k=\sum\limits_k kp_k+6m$, $\sum\limits_k(k-6)p_k+6m=0$, and 
\begin{equation}
3p_3+2p_4+p_5=6m+\sum\limits_{k\geqslant 7}(k-6)p_k
\end{equation}
Assume that  $P$ has no triangles and quadrangles. Then $p_5=6m+\sum_{k\geqslant 7}(k-6)p_k$
and $f_2=\sum\limits_{k\geqslant 5} p_k=6m+\sum_{k\geqslant 6}(k-5)p_k$.
Let $k_i$ be~the~number of edges of $i$-th defining $2$-face. Then 
$$
f_0=V_{4,r}+\sum_{i=1}^{m-r} k_i=f_2-f_3=5m+\sum_{k\geqslant 6}(k-5)p_k.
$$
Thus, 
\begin{gather*}
\sum_{i=1}^{m-r} (k_i-5)=5m+\sum_{k\geqslant 6}(k-5)p_k-5(m-r)-V_{4,r}=\\
5r-V_{4,r}+\sum_{k\geqslant 6}(k-5)p_k\geqslant
5r-V_{4,r}+\sum_{i=1}^{m-r} (k_i-5).
\end{gather*}
Then $V_{4,r}\geqslant 5r$. For $r=3$ we have $0\geqslant 15$, for $r=4$ we have $2\geqslant 20$, and for $r=5$ we have $5\geqslant 25$. In all cases this is a contradiction.

Since any right-angled hyperbolic $4$-polytope has no triangular and quadrangular $2$-faces,
it~does not admit Hamiltonian $\mathcal{C}(4,r)$-subcomplexes. Moreover, it follows from
the paper \cite{N82} by V.V.~Nikulin that for $n>4$ in the~hyperbolic $n$-space $\mathbb L^n$
there are no compact right-angled polytopes. This finishes the proof.
\end{proof}

\begin{proposition}\label{prop:IX}
If a $4$-polytope $P=Q\times I$ admits a Hamiltonian \linebreak $\mathcal{C}(4,r)$-subcomplex, then 
at least one of~the~defining $2$-faces of $P$ is a triangle. In particular, there  are~no~hyperelliptic 
$4$-manifolds $N(P,\Lambda)$ with geometries $\mathbb L^3\times\mathbb R$, $\mathbb L^2\times\mathbb R^2$,
and $\mathbb R ^4$ obtained by~Construction~\ref{con:hypgeom}.
\end{proposition}
\begin{proof}
Let $q$ be the number of facets of $Q$. Then $f_0(Q)=2(q-2)$, $m=f_3(P)=q+2$, and~$f_0(P)=2f_0(Q)=4(m-4)$.
Ley $k_i$ be the number of edges of the $i$-th defining $2$-face. 
We~have $f_0(P)=V_{4,r}+\sum_{i=1}^{m-r}k_i=4(m-4)=4m-16$.
Then the average number of edges in defining $2$-faces is
\begin{gather*}
\frac{\sum\limits_{i=1}^{m-r}k_i}{m-r}=\frac{4m-16-V_{4,r}}{m-r}=\frac{4(m-r)+4r-16-V_{4,r}}{m-r}=\\
=4-\frac{16+V_{4,r}-4r}{m-r}=
\begin{cases}
4-\frac{4}{m-3}& \text{if } r=3;\\
4-\frac{2}{m-4}& \text{if } r=4;\\
4-\frac{1}{m-5}& \text{if } r=5;
\end{cases}
\end{gather*}
This number is strictly less then $4$. Therefore, at least one defining face is a triangle.
\end{proof}
\begin{proposition}\label{prop:Sn}
For $n\geqslant 4$ the simplex $\Delta^n$ up to symmetries has a~unique proper Hamiltonian $\mathcal{C}(n,r)$-subcomplex
defined by a single $(n-2)$-face $\Delta^{n-2}$. For this subcomplex $r=n$. It~corresponds
to~a~hyperelliptic involution on the manifold $\mathbb{R}\mathcal{Z}_{\Delta^n}\simeq S^n$.  
In particular, the~geometry $\mathbb S^n$ arises in~Construction~\ref{con:hypgeom}. The branch set of~the covering $S^n\to S^n$ is the sphere~$S^{n-2}$. 
\end{proposition}
\begin{remark}
For $n=3$ the simplex $\Delta^3$ admits also a~Hamiltonian $\mathcal{C}(n,n-1)$-subcomplex corresponding
to a Hamiltonian cycle. 
\end{remark}
\begin{proof}[of Proposition~\ref{prop:Sn}]
Indeed, for $n\geqslant 4$ any two $(n-2)$-faces of $\Delta^n$ intersect. Also by~definition
any $(n-2)$-face of~$\Delta^n$ is~a~defining face a~Hamiltonian $\mathcal{C}(n,n)$-subcomplex.  The structure of the branch set follows from the~following fact.
\begin{lemma}\label{lem:RZF}
Let $F$ be one of the defining $(n-2)$-faces of~a~Hamiltonian \linebreak $\mathcal{C}(n,k)$-subcomplex $C\subset \partial P$. 
If the facets of $F$ correspond to different facets of $C$ (equivalently, if $F\cap F_i$ and $F\cap F_j$
are facets of $F$, then $F_i$ and $F_j$ lie in different facet of $C$), then in~the~branch set
of~the~covering $N(P,\widetilde{\Lambda}_C)\to S^n$ the~preimage of~$F$ 
is~a~disjoint union of $2^{k-1-m_F}$ copies of~$\mathbb R\mathcal{Z}_F$, where $m_F$ is~the~number 
of facets of $F$.
\end{lemma}
\begin{proof}
This follows from Proposition~\ref{prop:Ham2sh}.
\end{proof}
\end{proof}
\begin{proposition}\label{prop:Sn-1R}
For $n\geqslant 4$ the prism $\Delta^{n-1}\times I$ up to symmetries admits exactly 
$3$~Hamiltonian $\mathcal{C}(n,r)$-subcomplexes:
a~unique $\mathcal{C}(n,n+1)$-subcomplex and two $\mathcal{C}(n,n)$-subcomplexes. 
In~particular, the~geometry $\mathbb S^{n-1}\times \mathbb R$ arises in~Construction~\ref{con:hypgeom}.
For~the~$\mathcal{C}(n,n+1)$-subcomplex and the~$\mathcal{C}(n,n)$-subcomplexes 
the branch sets of the coverings $S^{n-1}\times S^1\to S^n$ and $N(P,\widetilde{\Lambda}_C)\to S^n$
are disjoint unions of two spheres $S^{n-2}$. 
\end{proposition}
\begin{proof}
Indeed, up~to~symmetries $\Delta^{n-1}\times I$ has two types of~$(n-2)$-faces:
$\Delta^{n-3}\times I$ and~$\Delta^{n-2}$. In~the~Hamiltonian subcomplex corresponding
to~$\Delta^{n-3}\times I$ the~facets corresponding to~$\Delta^{n-1}\times 0$ and $\Delta^{n-1}\times 1$
do not intersect. Therefore, this subcomplex is not equivalent to~$\mathcal{C}(n,r)$ for all~$r$.
The~face $\Delta^{n-2}$ corresponds to~a~Hamiltonian $\mathcal{C}(n,n+1)$-subcomplex.  

Any face $\Delta^{n-3}\times I$ intersects any face of the type $\Delta^{n-2}$. For $n\geqslant 5$
it also intersects any other face of~the~type $\Delta^{n-3}\times I$. Hence it can not be~a~defining face
of~a~Hamiltonian $\mathcal{C}(n,r)$-subcomplex. For $n=4$ and any face $I\times I$ there is a unique face $I\times I$ 
such that these faces are disjoint. In the corresponding Hamiltonian subcomplex the~facets corresponding
to $\Delta^3\times 0$ and $\Delta^3\times 1$ do not intersect. 
Therefore, this subcomplex is not equivalent to~$\mathcal{C}(n,r)$ for all~$r$.
Any face $\Delta^{n-2}\times \boldsymbol{x}$ intersects any other face $\Delta^{n-2}\times \boldsymbol{x}$.
Two faces $\Delta^{n-2}_1\times 0$ and $\Delta^{n-2}_2\times 1$ do not intersect. 
Up to symmetries there can be two types of such faces. In the first case $\Delta^{n-2}_1=\Delta^{n-2}_2$.
In~the~second case $\Delta^{n-2}_1\cap\Delta^{n-2}_2$ is an~$(n-3)$-simplex. Both these pairs of faces 
define Hamiltonian $\mathcal{C}(n,n)$-subcomplexes.  From the above argument there can not
be more than two defining faces. The structure of the branch set follows from Lemma~\ref{lem:RZF}.
\end{proof}

\begin{construction}[(Cutting off a face)]\label{con:cutF}
Let $P=\{\boldsymbol{a}_i\boldsymbol{x}+b_i\geqslant 0, i=1,\dots,m\}\subset \mathbb R^n$ be 
a~simple $n$-polytope and $G=F_{i_1}\cap\dots\cap F_{i_k}$ be its $(n-k)$-face, $k>1$.
Set $\boldsymbol{a}_G=\sum\limits_{s=1}^k\boldsymbol{a}_{i_s}$ and $b_G=\sum\limits_{s=1}^kb_{i_s}$.
Take the~halfspace $\mathcal{H}_{\geqslant \varepsilon}=\{\boldsymbol{a}_G\boldsymbol{x}+
b_G\geqslant \varepsilon\}$.
For small $\varepsilon>0$ the intersection $P_{G,\varepsilon}=P\cap \mathcal{H}_{\geqslant \varepsilon}$
is~a~simple $n$-polytope. It has an additional facet $F$ corresponding 
to~the~intersection of~$P$ with the hyperplane 
$\mathcal{H}_{\varepsilon}=\partial \mathcal{H}_{\geqslant \varepsilon}$.
We have $F\cap F_{j_1}\cap\dots\cap F_{j_l}\ne\varnothing$ if and only if $F_{j_1}\cap\dots\cap F_{j_l}$ 
is~a~face of $P$ intersecting $G$ and not lying in $G$.  That is, 
$F_{j_1}\cap\dots\cap F_{j_l}\cap F_{i_1}\cap\dots\cap F_{i_k}\ne\varnothing$ and $\{i_1,\dots,i_k\}\not\subset\{j_1,\dots,j_l\}$. 
\end{construction}
\begin{lemma}
The facet $F$ of $P_{G,\varepsilon}$ 
is combinatorially equivalent to $G\times \Delta^{k-1}$.
\end{lemma}
\begin{proof} 
Indeed, facets of $F$ are in bijection with
facets of $P$ intersecting $G$. These are $F_{i_1}$, $\dots$, $F_{i_k}$, and facets $F_j$ such that $F_j\cap G$
is a facet of $G$. The facets $F\cap F_{i_{a_1}}$, $\dots$, $F\cap F_{i_{a_l}}$, $F\cap F_{j_1}$, $\dots$, $F\cap F_{j_s}$ intersect
if and only if~$\{i_{a_1},\dots, i_{a_l}\}\ne \{i_1,\dots, i_k\}$, and 
$G\cap F_{j_1}\cap \dots\cap F_{j_s}=(G\cap F_{j_1})\cap\dots\cap (G\cap F_{j_s})\ne\varnothing$. Thus, 
$F\simeq G\times\Delta^{k-1}$.
\end{proof}
\begin{lemma}\label{lem:Ppart}
The other part $P\cap \mathcal{H}_{\leqslant \varepsilon}$ of~the~polytope $P$ 
is combinatorially equivalent to $G\times\Delta^k$. 
\end{lemma}
\begin{proof}
Indeed, by the same argument as above the facets of~the~polytope $P\cap \mathcal{H}_{\leqslant \varepsilon}$
are in~bijection with the facet $F$,  facets $F_{i_1}$, $\dots$, $F_{i_k}$, and facets $F_j$ such that $F_j\cap G$
is~a~facet of~$G$. Moreover, its collection of facets 
 $\{F, F_{i_{a_1}}, \dots, F_{i_{a_l}}, F_{j_1},\dots, F_{j_s}\}$ 
intersects if any only if $\{i_{a_1},\dots, i_{a_l}\}\ne \{i_1,\dots, i_k\}$ and $G\cap F_{j_1}\cap \dots\cap F_{j_s}=(G\cap F_{j_1})\cap\dots\cap (G\cap F_{j_s})\ne\varnothing$,
while for $\{F_{i_{a_1}}, \dots, F_{i_{a_l}}, F_{j_1},\dots, F_{j_s}\}$ the~criterion of~intersection 
is just $G\cap F_{j_1}\cap \dots\cap F_{j_s}\ne\varnothing$.
\end{proof}

\begin{lemma}\label{lem:defPG}
Any~$(n-2)$-face $F\cap F_{i_q}\simeq G\times\Delta^{k-2}$, $i_q\in\{i_1,\dots, i_k\}$, of $P_{G,\varepsilon}$ defines a Hamiltonian subcomplex
$C\simeq \partial P$.
\end{lemma} 
\begin{proof}
Let us rotate the hyperplane $\mathcal{H}_{\varepsilon}$ around the $(n-2)$-plane containing 
the~face $F\cap F_{i_q}$. The rotating plane $\mathcal{H}(\varphi)$ is defined by the formula 
$$
\cos \varphi(\boldsymbol{a}_G\boldsymbol{x}+b_G-\varepsilon)+\sin \varphi(\boldsymbol{a}_{i_q}\boldsymbol{x}+b_{i_q})=0.
$$ 
We have $\mathcal{H}(0)=\mathcal{H}_{\varepsilon}={\rm aff}(F)$ and $\mathcal{H}(\frac{\pi}{2})={\rm aff}(F_{i_q})$. 
Moreover, for small $\delta>0$ at each moment 
$-\delta\leqslant \varphi<\frac{\pi}{2}$ the polytope $P\cap \mathcal{H}_{\geqslant 0}(\varphi)$ has the same combinatorial
type $P_{\varepsilon}$ and for any 
$-\delta\leqslant \varphi_1<\varphi_2\leqslant\frac{\pi}{2}$ the polytope $P\cap \mathcal{H}_{\leqslant 0}(\varphi_1)
\cap \mathcal{H}_{\geqslant 0}(\varphi_2)$ has the same combinatorial type 
$G\times \Delta^k\simeq P\cap \mathcal{H}_{\leqslant \varepsilon}$. In particular, the polytope 
$P\cap \mathcal{H}_{\leqslant 0}(-\delta)\cap \mathcal{H}_{\geqslant 0}(0)$ is~combinatorially equivalent 
to~$P\cap\mathcal{H}_{\leqslant 0}(-\delta)\cap \mathcal{H}_{\geqslant 0}(\frac{\pi}{2})$. Then the piecewise
linear homeomorphism defined using the barycentric subdivisions preserves the face structures and is identical
on~$P\cap\mathcal{H}(-\delta)$. This homeomorphism together with the
identical mapping on $P\cap\mathcal{H}_{\geqslant 0}(-\delta)$ defines the~desired equivalence $C\simeq \partial P$.
\end{proof}
\begin{corollary}\label{cor:cutDn}
For any set $G_1$, $\dots$, $G_k$ of pairwise disjoint faces of the simplex $\Delta^n$
the~polytope obtained from $\Delta^n$ by cutting off all these faces by different hyperplanes
has a~Hamiltonian $\mathcal{C}(n,n+1)$-subcomplex given by a~set of $k$ defining $(n-2)$-faces lying 
in different cutting hyperplanes. 
\end{corollary}
\begin{example}
Any hyperplane separating a face $\Delta^k$ of the simplex $\Delta^n$ from the other
vertices leaves a face $\Delta^{n-k-1}$ on the other side. Then Lemma~\ref{lem:Ppart}
implies that if we cut off  $\Delta^k$ we~obtain combinatorially the polytope $\Delta^{n-k-1}\times\Delta^{k+1}$.
Moreover, the defining face $F\cap F_{i_q}$ from Lemma~\ref{lem:defPG} has the form $\Delta^{n-k-2}\times\Delta^k$.
\end{example}
\begin{corollary}\label{cor:spsq}
For any $p,q\geqslant 1$ the face $\Delta^{p-1}\times\Delta^{q-1}$ defines a~Hamiltonian
$\mathcal{C}(n,n+1)$-subcomplex in $P^n=\Delta^p\times\Delta^q$. In particular, for $p,q\geqslant 2$
the geometry $\mathbb S^p\times \mathbb S^q$ arises in~Construction~\ref{con:hypgeom}.
(For $p=1$ or $q=1$ we obtain the polytope $\Delta^{n-1}\times I$ and the geometry $\mathbb S^{n-1}\times\mathbb R$
already covered by Proposition~\ref{prop:Sn-1R}.) 
For $p,q\geqslant 2$~the branch set of~the~coverings $S^p\times S^q\to S^n$ 
is~homeomorphic to~$S^{p-1}\times S^{q-1}$. 
\end{corollary}
\begin{remark}
The structure of the branch set follows from Lemma~\ref{lem:RZF}. Moreover, 
this implies that in~Corollary~\ref{cor:cutDn} the branch set is a disjoint union of the products $S^{n-k-2}\times S^k$
corresponding to faces $\Delta^k$, $0\leqslant k\leqslant n-2$.
\end{remark}
\begin{proposition}\label{prop:Dn-1}
If an $n$-polytope $P$ has a~Hamiltonian $\mathcal{C}(n,n+1)$-subcomplex such
that its~defining faces $\{M_1,\dots, M_s\}$ do not intersect a facet $F_i\simeq \Delta^{n-1}$,
then for any $k\geqslant 1$ the~polytope $\Delta^k\times P$ has a~Hamiltonian 
$\mathcal{C}(n+k,n+k+1)$-subcomplex with defining faces $\Delta^{k-1}\times F_i$ and $\Delta^k\times M_1$,
$\dots$, $\Delta^k\times M_s$. 
\end{proposition}
\begin{proof}
Indeed, the faces  $\Delta^k\times M_1$,
$\dots$, $\Delta^k\times M_s$ define the subcomplex equivalent to $\partial (\Delta^k\times \Delta^n)$,
and the face $\Delta^{k-1}\times F_i\simeq \Delta^{k-1}\times\Delta^{n-1}$ transforms it to $\partial \Delta^{n+k}$.
\end{proof}
\begin{corollary}\label{cor:pqI}
For any $p,q\geqslant 1$ the polytope $P^n=\Delta^p\times\Delta^q\times I$
has a Hamiltonian $C(n, n+1)$-subcomplex defined by two faces $\Delta^{p-1}\times\Delta^q\times \{0\}$
and $\Delta^p\times\Delta^{q-1}\times \{1\}$. 
In particular, the~geometries $\mathbb S^p\times \mathbb S^q\times \mathbb R$, $p,q\geqslant 2$, 
$\mathbb S^p\times \mathbb R^2$, $p\geqslant 2$, and $\mathbb R^3$ arise in~Construction~\ref{con:hypgeom}.
The~branch set of~the~covering  $N(P,\widetilde{\Lambda}_C)\to S^n$ is  $S^{p-1}\times S^q\sqcup S^p\times S^{q-1}$.
\end{corollary}
\begin{proof}
Indeed, $\Delta^{q-1}\times\{1\}\cap \Delta^q\times \{0\}=\varnothing$. Hence, we can apply Proposition~\ref{prop:Dn-1}
to $P=\Delta^q\times I$. The structure of the branch set follows from Lemma~\ref{lem:RZF}.
\end{proof}

\begin{proposition}\label{prop:nn}
If an $n$-polytope $P$ has a~Hamiltonian $\mathcal{C}(n,n+1)$-subcomplex $C$ such
that its defining faces $\{M_1,\dots, M_s\}$ do not intersect an $(n-2)$-face $G\subset P$, then $G\simeq \Delta^{n-2}$,
and the~faces $\{M_1,\dots, M_s, G\}$ define  a~Hamiltonian $\mathcal{C}(n,n)$-subcomplex in~$\partial P$.
\end{proposition}
\begin{proof}
By Corollary \ref{cor:GneMq} $G$ is~an~$(n-2)$-face  of $C$, hence $G\simeq\Delta^{n-2}$ and 
it defines in $C$ the~Hamiltonian $\mathcal{C}(n,n)$-subcomplex. Then  the faces $\{M_1,\dots, M_s, G\}$ 
define  the~corresponding Hamiltonian $\mathcal{C}(n,n)$-subcomplex in~$\partial P$.
\end{proof}
\begin{corollary}\label{cor:pk}
Let $\Delta^2={\rm conv}\{v_1,v_2,v_3\}$. Then the faces ${\rm conv}\{v_1,v_2\}\times \Delta^{n-3}$ and 
$v_3\times\Delta^{n-2}$ define a~Hamiltonian
$\mathcal{C}(n,n)$-subcomplex $C$~in $\Delta^2\times \Delta^{n-2}$ for $n\geqslant 3$. The~branch set
of~the~covering $N(P,\widetilde{\Lambda}_C)\to S^n$ is a disjoint union of $S^1\times S^{n-3}$ and $S^{n-2}$ .
\end{corollary}
\begin{proof}
The structure of the branch set over $v_3\times\Delta^{n-2}$ follows from Lemma~\ref{lem:RZF}. 
For~the~face $F={\rm conv}\{v_1,v_2\}\times \Delta^{n-3}$ its facets corresponding to $\Delta^{n-3}$ 
lie in different facets of $C$, while the~facets $v_1\times \Delta^{n-3}$ and $v_2\times \Delta^{n-3}$
lie in~the~same facet of $C$ different from the above facets. Hence, by Proposition~\ref{prop:Ham2sh}
the~branch set over $F$ is homeomorphic to $S^1\times S^{n-3}$. 
\end{proof}
\begin{corollary}\label{cor:SxX2}
Let $P_k$ be a~$k$-gon, $k\geqslant 3$. Then for $p\geqslant 1$ the polytope $P^n=\Delta^p\times P_k$
admits Hamiltonian $\mathcal{C}(n,n)$- and $\mathcal{C}(n,n+1)$-subcomplexes $C_n$ and $C_{n+1}$. 
In particular, taking a~right-angled triangle in $\mathbb S^2$, a~square in~$\mathbb R^2$ and 
a~right-angled $k$-gon, $k\geqslant 5$, in $\mathbb L^2$ we see that the~geometries $\mathbb S^p\times \mathbb S^2$,
$\mathbb S^p\times \mathbb R^2$
and $\mathbb S^p\times \mathbb L^2$, $p\geqslant 2$, arise in~Construction~\ref{con:hypgeom}.
The~branch set of~the~covering $N(P,\widetilde{\Lambda}_{C_{n+1}})\to S^n$ 
is a disjoint union of $2(k-3)$ copies of $S^{n-2}$ and one~$S^{n-3}\times S^1$, and
for the~covering $N(P,\widetilde{\Lambda}_{C_n})\to S^n$ it~is~a~disjoint union of~$(k-2)$ copies of~$S^{n-2}$
and one~$S^{n-3}\times S^1$.
\end{corollary}
\begin{proof}
Indeed, let $P_k={\rm conv}\{v_1,\dots, v_k\}$. Then the vertices $v_1$, $\dots$, $v_{k-3}$ 
define a~Hamiltonian~$\mathcal{C}(2,3)$-subcomplex in $P_k$, and the edge ${\rm conv}\{v_{k-2},v_{k-1}\}$
does not intersect these vertices. 
By~Proposition~\ref{prop:Dn-1} the faces $\Delta^p\times v_1$, $\dots$, $\Delta^p\times v_{k-3}$,
$\Delta^{p-1}\times {\rm conv}\{v_{k-2},v_{k-1}\}$ define a~Hamiltonian $\mathcal{C}(n,n+1)$-subcomplex $C_{n+1}$.
By Proposition~\ref{prop:nn} the addition of the face $\Delta^p\times v_k$ defines
a~Hamiltonian $\mathcal{C}(n,n)$-subcomplex $C_n$. For each defining face~of~$C_{n+1}$
its facets correspond to different facets of $C_{n+1}$, hence the structure of the branch set 
follows from Lemma~\ref{lem:RZF}. For defining faces $\Delta^p\times v_j$ of $C_n$ the same
argument works. For the defining face $\Delta^{p-1}\times {\rm conv}\{v_{k-2},v_{k-1}\}$ of~$C_n$
its facets corresponding to $\Delta^{p-1}$ lie in~different facets of~$C_n$, while the facets 
$\Delta^{p-1}\times v_{k-2}$ and $\Delta^{p-1}\times v_{k-1}$ lie in~the~same facet of~$C_n$
different from the~above facets. Hence,  by Proposition~\ref{prop:Ham2sh}
the~branch set over this facet is homeomorphic to $S^{p-1}\times S^1$. 
\end{proof}
\begin{example}\label{ex:l2l2}
On Fig.~\ref{5x5} we show a set of five disjoint quadrangles  defining 
a~Hamiltonian $\mathcal{C}(4,5)$-subcomplex in~$P_5\times P_5$. In particular,
the geometry $\mathbb L^2\times\mathbb L^2$ arises in Construction~\ref{con:hypgeom}.
It follows from Lemma~\ref{lem:RZF} that~the branch set of~the~covering $N(P,\widetilde{\Lambda}_C)\to S^4$
is~a~disjoint union of~five tori $\mathbb T^2=S^1\times S^1$.
Comparing Fig.~\ref{5x5}
and \cite[Fig. 13 and 17]{M24} it can be proved that these five tori coincide with the Ivan\v{s}i\'{c}  link \cite{I04,I12},
and the manifold $N(P,\widetilde{\Lambda}_C)$ coincides with its double branched covering space $W$ from 
\cite[Section 2.2]{M24}, where on this manifolds the geometric structure $\mathbb L^2\times \mathbb L^2$ is also defined.
The complement to the Ivan\v{s}i\'{c} link in $S^4$ has a complete hyperbolic structure obtained
by glueing $32$ copies of the right-angled hyperbolic  $4$-polytope of finite volume 
$P^4\subset \mathbb L^4$. It has $5$ ideal and $5$ finite vertices \cite[Section 1.2]{M24}. The dual polytope
coincides with the hypersimplex $\Delta(5,2)$. Each facet of  $P$ is a triangular bipyramid and has three ideal and
two finite vertices. The Ivan\v{s}i\'{c} link  is a $4$-dimensional generalisation of the Borromean rings: each
two tori form a trivial link, but each three do not \cite[Theorem 1]{M24}. The author is grateful to Vladimir
Gorchakov for pointing out a connection between this example and the work \cite{M24}.
\begin{figure}[h]
\begin{center}
\includegraphics[width=0.8\textwidth]{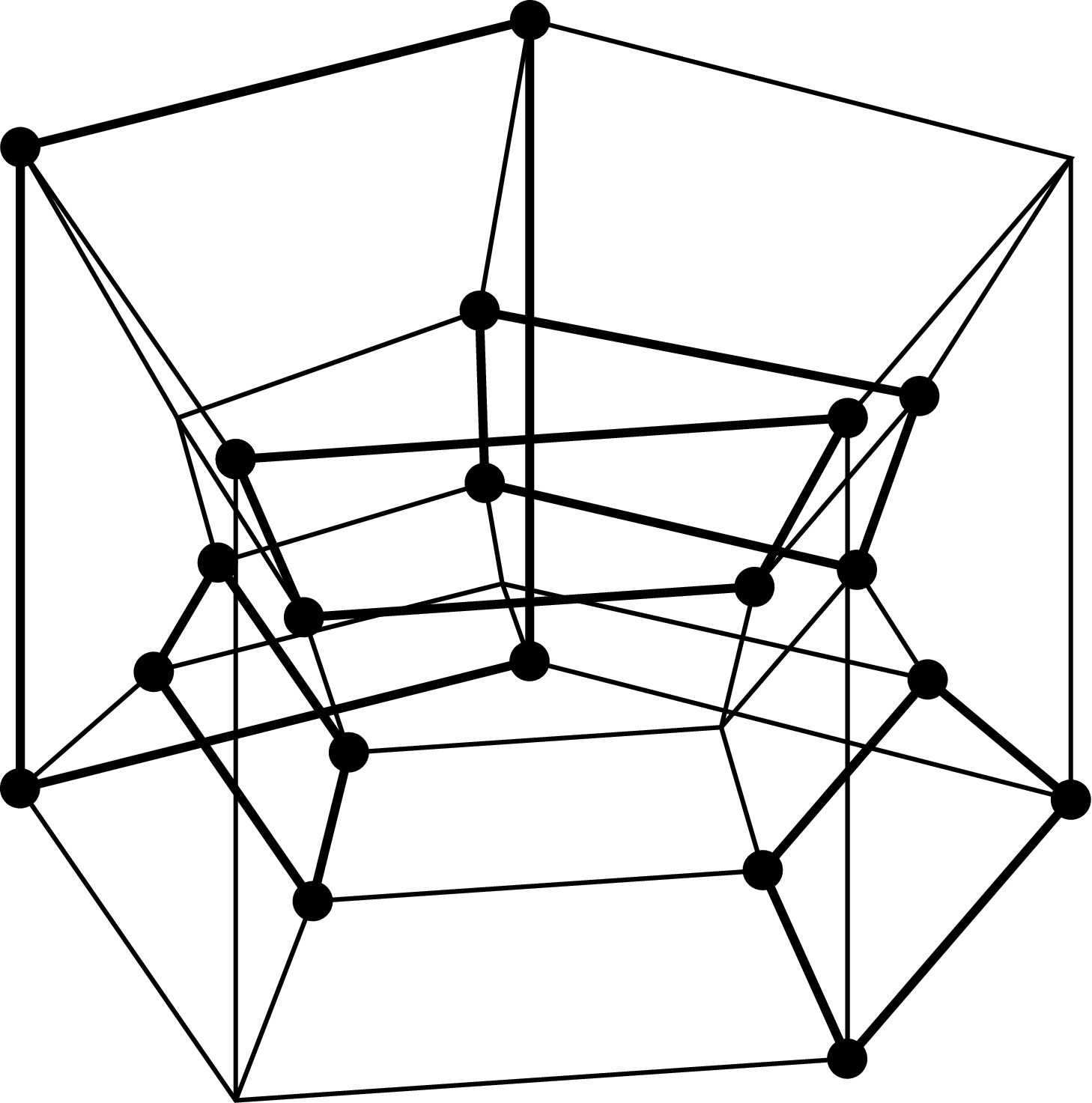}
\end{center}
\caption{A set of disjoint quadrangles defining a~Hamiltonian $\mathcal{C}(4,5)$-subcomplex 
in~$P_5\times P_5$}\label{5x5}
\end{figure}
\end{example}
\begin{remark} 
In \cite{E25} there is a construction of a $3$-dimensional generalisation
of the Borromean rings: each Hamiltonian theta-subgraph $\Gamma$ in the $1$-skeleton of 
a~compact right-angled hyperbolic $3$-polytope $P$ 
such that each edge of $P$ not lying in $\Gamma$ connects vertices on different paths of $\Gamma$ corresponds
to a link $L$ of circles in $S^3$ with the complement having a complete hyperbolic structure. Any two 
circles of $L$ form a trivial link, but $L$ contains the Borromean rings and therefore is nontrivial.
\end{remark}

In dimension $n=4$ all the products of Euclidean, spherical and hyperbolic geometries (including
these geometries) are $10$ geometries: 
$\mathbb S^4$, $\mathbb R^4$, $\mathbb L^4$, $\mathbb S^3\times \mathbb R$,
$\mathbb S^2\times\mathbb S^2$, $\mathbb S^2\times\mathbb R^2$, $\mathbb S^2\times \mathbb L^2$,
$\mathbb L^3\times \mathbb R$, $\mathbb L^2\times \mathbb R^2$, and $\mathbb L^2\times\mathbb L^2$.
\begin{theorem}\label{th:geomn=4}
In dimension $n=4$ among all products of Euclidean, spherical and hyperbolic geometries (including
these geometries) the geometries $\mathbb S^4$, $\mathbb S^3\times \mathbb R$, 
$\mathbb S^2\times \mathbb S^2$, $\mathbb S^2\times \mathbb R^2$, $\mathbb S^2\times \mathbb L^2$,
and $\mathbb L^2\times \mathbb L^2$ admit hyperelliptic manifolds build by Construction~\ref{con:hypgeom},
and the geometries $\mathbb R^4$, $\mathbb L^4$, $\mathbb L^3\times \mathbb R$, $\mathbb L^2\times\mathbb R^2$
do not admit.
\end{theorem}
\begin{proof}
The cases of $\mathbb S^4$, $\mathbb S^3\times \mathbb R$,
$\mathbb S^2\times \mathbb S^2$, $\mathbb S^2\times \mathbb R^2$, $\mathbb S^2\times \mathbb L^2$, $\mathbb L^2\times \mathbb L^2$
are covered by Propositions~\ref{prop:Sn}, \ref{prop:Sn-1R}, Corollaries \ref{cor:spsq}, \ref{cor:SxX2}, and Example \ref{ex:l2l2}.

The cases of $\mathbb R^4$, $\mathbb L^4$, $\mathbb L^3\times \mathbb R$, $\mathbb L^2\times\mathbb R^2$ are covered
by Corollary~\ref{cor:In}, and Propositions~\ref{prop:Ln}, \ref{prop:IX}.
\end{proof}

\section{Acknowledgements}
The study has been funded within the framework of the HSE University Basic Research Program.

The author is grateful to V.\,M.~Buchstaber for his permanent attention, to A.\,D.~Mednykh 
for~motivating discussions, in particular for the~suggestion 
to consider geometries arising on~hyperelliptic $4$-manifolds $N(P,\Lambda)$ and the~corresponding branch sets,
and to D.\,V.~Gugnin, A.\,Yu.~Vesnin, and V.\,Yu.~Gorchakov for~useful discussions.

\end{document}